\documentclass{amsart}
\usepackage{amsmath}
\usepackage{amssymb}
\usepackage{amsopn}
\usepackage{amsfonts}
\usepackage{epsfig}
\usepackage{enumitem} 
\usepackage{hyperref}
\usepackage{latexsym}
\usepackage{pictex}
\usepackage{xcolor}

\usepackage[utf8]{inputenc}

\newtheorem{theorem}{Theorem}[section]
\newtheorem{lemma}[theorem]{Lemma}
\newtheorem{proposition}[theorem]{Proposition}

\newtheorem{conjecture}[theorem]{Conjecture}
\newtheorem{remark}[theorem]{Remark}

\title[$\alpha$-continued fraction expansions with odd partial quotients]{Natural extensions and entropy
of $\alpha$-continued fraction expansions with odd partial quotients}

{\author{Yusuf Hartono}
\address{Department of Mathematics Education, Sriwijaya University (Unsri), Jalan Raya Palembang-Prabumulih Km 32 Indralaya 30662 , Indonesia}
\email{yhartono@unsri.ac.id}}
{\author{Cor Kraaikamp}
\address{Delft University of Technology, EWI (DIAM), Mekelweg 4, 2628 CD Delft, the
Netherlands}
\email{c.kraaikamp@tudelft.nl}}
{\author{Niels Langeveld}
\address{Leiden University, Mathematical Institute, PO Box 9512, 2399RA Leiden, the
Netherlands} 
\email{n.d.s.langeveld@math.leidenuniv.nl}
\address{Montan University Leoben, Franz Josef-Stra\upshape{\ss}e 18, 8700 Leoben, Austria} \email{niels.langeveld@unileoben.ac.at}}
{\author{Claire Merriman}
\address{The Ohio State University, 231 West 18th Ave, Columbus, Ohio}
\email{merriman.72@osu.edu}}

\subjclass{Primary 28D05, 11K50}
\keywords{odd continued fractions, natural extensions, matching}%

\begin{document}

\begin{abstract}
In~\cite{[BM]}, Boca and the fourth author of this paper introduced a new class of continued fraction expansions \emph{with odd partial quotients}, parameterized by a parameter $\alpha\in [g,G]$, where $g=\tfrac{1}{2}(\sqrt{5}-1)$ and $G=g+1=1/g$ are the two \emph{golden mean numbers}. Using operations called \emph{singularizations} and \emph{insertions} on the partial quotients of the odd continued fraction expansions under consideration, the natural extensions from~\cite{[BM]} are obtained, and it is shown that for each $\alpha,\alpha^*\in [g,G]$ the natural extensions from~\cite{[BM]} are metrically isomorphic. An immediate consequence of this is, that the entropy of all these natural extensions is equal for $\alpha\in [g,G]$, a fact already observed in~\cite{[BM]}. Furthermore, it is shown that this approach can be extended to values of $\alpha$ smaller than $g$, and that for values of $\alpha \in [\tfrac{1}{6}(\sqrt{13}-1), g]$ all natural extensions are still isomorphic. In the final section of this paper further attention is given to the entropy, as function of $\alpha\in [0,G]$. It is shown that in any neighborhood of $0$ we can find intervals on which the entropy is decreasing, intervals on which the entropy is increasing and intervals on which the entropy is constant. In order to prove this we use a phenomena called matching.
\end{abstract}

\maketitle

\newtheorem{Theorem}{Theorem}[section]
\newtheorem{Proposition}{Proposition}[section]
\newtheorem{Lemma}{Lemma}[section]
\newtheorem{Corollary}{Corollary}[section]
\newtheorem{Definition}{Definition}[section]
\newcommand{\R}{\mathbb R} \newcommand{\Z}{\mathbb Z}
\newcommand{\N}{\mathbb{N}} \newcommand{\Q}{\mathbb{Q}}
\newcommand{\ds}{\displaystyle}
\newcommand{\m}{\mathbf}

\section{Introduction}\label{sec:introduction}
\noindent It is well known that every real number $x$ can be written as a finite (in case $x\in\Q$) or infinite (regular)
continued fraction of the form
\begin{equation}\label{eq:RCF}
x=a_0+\frac{\displaystyle 1}{\displaystyle a_1+\frac{1}{a_2+\ddots
+\displaystyle \frac{1}{a_n+\ddots}}} = [a_0;a_1,a_2,\dots ,a_n,\dots ],
\end{equation}
where $a_0\in\Z$ such that $x-a_0\in [0,1)$, and $a_n\in\N$ for $n\geq 1$. Such a \emph{regular continued fraction expansion} (RCF) of $x$ is unique if and only if $x$ is irrational; in case $x\in\Q$ one has two finite expansions of the form~(\ref{eq:RCF}).\smallskip\

Apart from the RCF there are very many continued fraction algorithms, such as Nakada's $\alpha$-expansions from~\cite{[BKS],[N]} and the Ito-Tanaka $\alpha$-expansions from~\cite{[TI]} (both classes generalizing the RCF), the Rosen fractions (generalizing the \emph{nearest integer continued fraction}; see \cite{[KSS],[R]}, and the continued fractions with either \emph{odd} or \emph{even} partial quotients; see~\cite{[S1],[S2],[Rie],[Se1],[Se2],[HK]}.\medskip\

Recently, Boca and the fourth author introduced in~\cite{[BM]} a new class of continued fraction expansions: $\alpha$-continued fraction expansions with odd partial quotients. These $\alpha$-expansions are reminiscent of Nakada's $\alpha$-expansions from~\cite{[N]}, and are studied for a certain range of the parameter (in case of Nakada: $\alpha \in [\tfrac{1}{2},1]$, in case of Boca and the fourth author: $\alpha\in [g,G]$, where $g=(\sqrt{5}-1)/2$ and $G=g+1 = 1/g = (\sqrt{5}+1)/2$ are the two \emph{golden mean} numbers). For both types of $\alpha$-expansions the respective authors show that the underlying dynamical system is ergodic, find its natural extension and obtain the entropy for each $\alpha$ under consideration. For $\alpha = 0.9g$, Boca and the fourth author use a simulation to give an impression of the natural extension of this $\alpha$-expansion, which seems to be quite different from the case where $\alpha \in [g,G]$.\medskip\

In this paper we obtain the results from~\cite{[BM]} in a very different way,  enabling us to obtain also the underlying dynamical system for $\alpha \in [ \frac{ \sqrt{13}-1}{6} ,g)$. In particular, we will see that our construction yields that all natural extensions under consideration are metrically isomorphic, and therefore confirming the result of Boca and the fourth author for $\alpha \in [g,G]$. We would like to remark that for $\alpha \leq 1$ we do not have that all branches are expansive. In~\cite{[S1]} Schweiger proved for $\alpha=1$ that the corresponding system is ergodic, so by using the isomorphism we construct it follows that for all $\alpha\in [ \frac{ \sqrt{13}-1}{6},G]$ the corresponding dynamical system is ergodic. \medskip

In Section~\ref{sec:ent} we turn our attention to entropy as a function of $\alpha$. We obtain a result analogous to a  result of Nakada and Natsui from~\cite{[NN]} which states that in any neighborhood of $0$ we can find intervals on which the entropy is decreasing, intervals on which the entropy is increasing and intervals on which the entropy is constant. In order to prove this we use a phenomena called matching.

\section{Odd continued fraction expansions}\label{sec:oddexpansions}
Let $\alpha\in [g,G]$, and set $I_{\alpha} = [\alpha -2,\alpha )$. Then for $x\in I_{\alpha}$ the $\alpha$-continued fraction map $T_{\alpha}: I_{\alpha}\to I_{\alpha}$ is defined in~\cite{[BM]} as
\begin{equation}\label{eq:alphamap}
T_{\alpha}(x) = \frac{\varepsilon (x)}{x} - d_{\alpha}(x),\quad \text{if $x\in I_{\alpha}\setminus \{ 0 \}$},
\end{equation}
and $T_{\alpha}(0)=0$, where $\varepsilon (x) = \text{sign}(x)$ and
$$
d_{\alpha}(x) = 2\left\lfloor \frac{1}{2|x|} + \frac{1-\alpha}{2}\right\rfloor +1,\quad \text{if $x\in I_{\alpha}\setminus \{ 0 \}$}.
$$
For $x\in I_{\alpha}\setminus \{ 0\}$, the map $T_{\alpha}$ ``generates'' a continued fraction in the following way: if $T_{\alpha}^{n-1}(x) \neq 0$ for $n\geq 1$, set $\varepsilon_n = \text{sign}(T_{\alpha}^{n-1}(x))$ and $a_n = d_{\alpha}(T_{\alpha}^{n-1}(x))$ (note that $T_{\alpha}^{n-1}(x)=0$ for some $n\geq 1$ if and only if $x\in \Q$). Then from~(\ref{eq:alphamap}) it follows that:
\begin{equation}\label{eq:fractionforn=1}
x = \frac{\varepsilon_1}{a_1+T_{\alpha}(x)}.
\end{equation}
Since in general for $n\geq 1$,
\begin{equation}\label{eq:definitionmapattimen}
T_{\alpha}^n(x) = \frac{\varepsilon_n}{T_{\alpha}^{n-1}(x)} - a_n
\end{equation}
we find that
\begin{equation}\label{eq:fractionforn}
T_{\alpha}^{n-1}(x) = \frac{\varepsilon_n}{a_n+T_{\alpha}^n(x)}.
\end{equation}
But then from~(\ref{eq:fractionforn=1}) and~(\ref{eq:fractionforn}) we find that, if $T_{\alpha}^{n-1}(x)\neq 0$,
\begin{equation}\label{eq:alphaexpansionofx}
x = \frac{\varepsilon_1}{a_1+\displaystyle{\frac{\varepsilon_2}{a_2+T_{\alpha}^2(x)}}} = \cdots
= \frac{\varepsilon_1}{a_1+\displaystyle{\frac{\varepsilon_2}{a_2+ \ddots + \displaystyle{\frac{\varepsilon_{n}}{a_{n}+T_{\alpha}^n(x)}}}}}.
\end{equation}
If $x$ is rational, there is an $n\in\N$ such that $T_{\alpha}^n(x)=0$, and the expansion of $x$ in~(\ref{eq:alphaexpansionofx}) is finite. If $x$ is irrational it follows from~(\ref{eq:alphamap}) that $T_{\alpha}^n(x)\neq 0$ for all $n\geq 1$. In this case, deleting $T_{\alpha}^n(x)$ from~(\ref{eq:alphaexpansionofx}) yields a so-called \emph{convergent}:
\begin{equation}\label{eq:convergentofx}
\frac{p_{\alpha ,n}}{q_{\alpha, n}} = \frac{\varepsilon_1}{a_1+\displaystyle{\frac{\varepsilon_2}{a_2+ \ddots + \displaystyle{\frac{\varepsilon_{n}}{a_{n}}}}}},
\end{equation}
where we assume that $p_{\alpha ,n}$ and $q_{\alpha ,n} > 0$ are integers in their lowest terms. Similar to the regular continued fraction (see e.g.~\cite{ [DK],[IK]}) we obtain that
$$
x=\lim_{n\to\infty} \frac{p_{\alpha ,n}}{q_{\alpha, n}}.
$$
In view of this we write
$$
x = \frac{\varepsilon_1}{a_1+\displaystyle{\frac{\varepsilon_2}{a_2+ \ddots + \displaystyle{\frac{\varepsilon_{n}}{a_{n}+\ddots}}}}},
$$
which is denoted as $x = [0;\varepsilon_1/a_1,\varepsilon_2/a_2, \dots , \varepsilon_n/a_n,\dots ]$.\footnote{In Section \ref{sec:ent} we use the notation $x=[0;\varepsilon_1a_1,\varepsilon_2a_2,\dots,\varepsilon_na_n,\dots]$ in order to save space. Since for $i\geq 1$ the partial quotients are odd positive integers the meaning is still straightforward.} In later sections we will use that the $p_{\alpha ,n}$ and $q_{\alpha ,n}$ satisfy\footnote{We will often suppress the $\alpha$ in $p_{\alpha ,n}$ and $q_{\alpha ,n}$.} recurrence relations, given by:
\begin{eqnarray}\label{eq:pnrecurrencerelation}
p_{\alpha ,-1} := 1 & p_{\alpha ,0} := 0 & p_{\alpha ,n} = a_np_{\alpha ,n-1} + \varepsilon_{n}p_{\alpha ,n-2} \label{eq:pnrecurrencerelation} \\
q_{\alpha ,-1} := 0 & q_{\alpha ,0} := 1 & q_{\alpha ,n} = a_nq_{\alpha ,n-1} + \varepsilon_{n}q_{\alpha ,n-2} \label{eq:qnrecurrencerelation} .
\end{eqnarray}

The case $\alpha = 1$, which will be the starting case of our investigations, has previously been studied by Schweiger in~\cite{[S1],[S2]} and Rieger in~\cite{[Rie]}. In particular, Schweiger obtained in~\cite{[S1]} the natural extension of the \emph{continued fraction with odd partial quotients} (oddCF). This natural extension is on $[0,1)\times [-g^2,G]$, and is -- as we will see shortly -- an isomorphic copy of the system found by Boca and the fourth author. It is nice to note, that Schweiger's natural extension has Rieger's \emph{grotesque continued fraction} (GCF) as the inverse of the second coordinate map (with other words, Rieger's GCF are the dual continued fraction expansions of Schweiger's oddCF), which is the case $\alpha = G$ in~\cite{[BM]}. In fact, if we would consider in Schweiger's natural extension the inverse of his map, and exchange the order of the coordinates, we get the system for $\alpha = G$ as obtained in~\cite{[BM]}. Sebe also obtained the natural extension of the GCF; see~\cite{[Se2]}. We will also use this system; see Subsection~\ref{subsec:grotesquesystem}.

\subsection{The continued fraction expansion with odd partial quotients}\label{subsec:odcfs}
The case $\alpha =1$ is the case of Schweiger's \emph{continued fractions with odd partial quotients}. In~\cite{[S2]} Schweiger defines his continued fraction map $T:[0,1)\to [0,1]$ as follows: let
$$
B(+,k) = \left( \frac{1}{2k}, \frac{1}{2k-1}\right] ,\quad \text{for $k=1,2,\dots$},
$$
and
$$
B(-,k) = \left( \frac{1}{2k-1}, \frac{1}{2k-2}\right] ,\quad \text{for $k=2,3,\dots$},
$$
then the map
$$
T(x) = \varepsilon \left( \frac{1}{x} - (2k - 1)\right) ,\quad \text{for $x\in B(\varepsilon , k)$ and $\varepsilon = \pm 1$},
$$
yields the oddCF-expansion of $x$. Schweiger showed that the map $\mathcal{T}: [0,1]\times [-g^2,G]\to [0,1]\times [-g^2,G]$, defined by
$$
\mathcal{T}(x,y) = \left( T(x), \frac{\varepsilon}{a+y}\right) ,
$$
where $\varepsilon = \pm 1$ and $a=2k-1$ are such that $x\in B(\varepsilon ,k)$, is the natural extension map of $T$, and that the system
\begin{equation}\label{eq:naturalextensionOddCF}
( \Omega = [0,1]\times [-g^2,G], \mathcal{B}, \bar{\mu}, \mathcal{T})
\end{equation}
is an ergodic system. Here $\mathcal{B}$ is the collection of Borel sets of $\Omega$, and $\bar{\mu}$ is a $\mathcal{T}$-invariant probability measure on $\Omega$ with density $(3\log G)^{-1}(1+xy)^{-2}$. Now define the map $M:\Omega\to\R^2$ by
$$
M(x,y) = \begin{cases}
(x,y),   & \text{if $y\geq 0$};\\
(-x,-y), & \text{if $y<0$}.
\end{cases}
$$
Now set $\Omega_1=M(\Omega )$, and let $\mathcal{T}_1:\Omega_1\to\Omega_1$ be defined as:
\begin{equation}\label{eq:narturalextensionmapalpha=1}
\mathcal{T}_1(x,y) = M(\mathcal{T}(M^{-1}(x,y))),\quad \text{for $(x,y)\in \Omega_1$}.
\end{equation}
Since $M$ preserves the $\bar{\mu}$ measure, and since
\begin{equation}\label{eq:naturalextensionalpha=1}
\mathcal{T}_1(x,y) = \left( T_1(x), \frac{1}{d_1(x)+\varepsilon (x)y}\right) ,
\end{equation}
we see that the dynamical system
$$
(\Omega_1,\mathcal{B}_1,\bar{\mu}_1, \mathcal{T}_1)
$$
is metrically isomorphic with Schweiger's system $(\Omega, \mathcal{B},\bar{\mu},\mathcal{T})$ from~(\ref{eq:naturalextensionOddCF}). It is exactly this ergodic system which was obtained in~\cite{[BM]} for the case $\alpha =1$. Note that $\mathcal{B}_1$ is the collection of Borel subsets of $\Omega_1$, and that on $\Omega_1$ the $\mathcal{T}_1$-invariant probability measure $\bar{\mu}_1$ has density $(3\log G)^{-1}(1+xy)^{-2}$. For the remainder of this section we will work with this dynamical system. See Figure~\ref{fig:OmegaOmega1} for both planar regions $\Omega$ and $\Omega_1$.

\begin{figure}[ht]
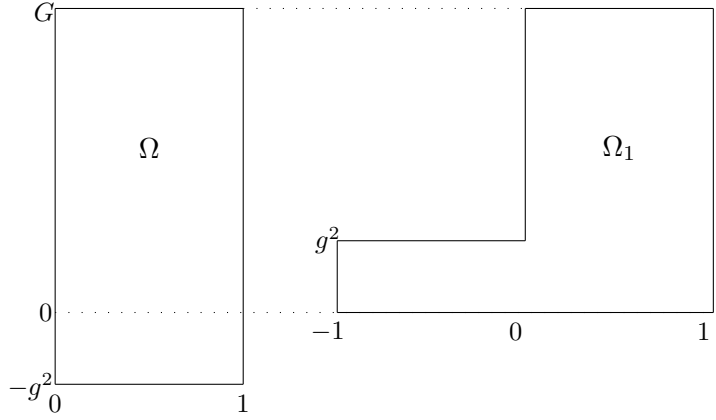

\beginpicture 

\setcoordinatesystem units <0.125cm,0.125cm> 
\setplotarea x from -10 to 60, y from -10 to 40

\put {$-g^2$} at -2.6 -7.9
\put {$G$} at -1.15 32
\put {$0$} at -1 0
\put {$0$} at 0 -9.7
\put {$1$} at 20 -9.7
\put {\large $\Omega$} at 10 17.5
\putrule from 0 -7.63932 to 20 -7.63932
\putrule from 20 -7.63932 to 20 32.36068
\putrule from 0 32.36068 to 20 32.36068
\putrule from 0 -7.63932 to 0 32.36068

\putrule from 30 0 to 70 0
\putrule from 70 0 to 70 32.36068
\putrule from 50 32.36068 to 70 32.36068
\putrule from 50 7.63932 to 50 32.36068
\putrule from 30 7.63932 to 50 7.63932
\putrule from 30 0 to 30 7.63932

\put {$g^2$} at 29 7.7
\put {$-1$} at 29 -2
\put {$0$} at 49 -2
\put {$1$} at 69 -2
\put {\large $\Omega_1$} at 60 17.5

\setdots
\putrule from 0 0 to 70 0
\putrule from 20 32.36068 to 50 32.36068
\endpicture
\caption{The domains $\Omega$ and $\Omega_1$.}
\label{fig:OmegaOmega1}
\end{figure}

Later in this section we will show how the planar natural extensions found by Boca and the fourth author for $\alpha\in [g,1)$ can be obtained from $\Omega_1$ via some simple operations on the partial quotients of any $x\in I_1\setminus \{0\}$ (see Subsection~\ref{subsec:singularizationinsertions1}), while we will deal with the case $\alpha \in (1,G]$ in Subsection~\ref{subsec:grotesquesystem}. In Section~\ref{subsec:smalleralpha} we will find the planar natural extension in case $\alpha\in [ \frac{\sqrt{13}-1 }{6} ,g]$ which is the more difficult case.\medskip\

Although we still need to determine the domain $\Omega_{\alpha}$, the natural extension map $\mathcal{T}_{\alpha}:\Omega_{\alpha}\to\Omega_{\alpha}$ is for each $\alpha$ given by:
\begin{equation}\label{eq:naturalextensionalphageneral}
\mathcal{T}_{\alpha}(x,y) = \left( T_{\alpha}(x), \frac{1}{d_{\alpha}(x)+\varepsilon (x)y}\right) ;
\end{equation}
see~(\ref{eq:naturalextensionalpha=1}) for the case $\alpha =1$. For $x\in I_{\alpha}\setminus \{ 0\}$, and for $n\geq 1$ for which $T_{\alpha}^n(x)\neq 0$ (so this is for all $n\geq 1$ for all irrational $x\in I_{\alpha}$) we set:
$$
(t_n,v_n) = \mathcal{T}^n_{\alpha}(x,0).
$$
Now if the $\alpha$-expansion of $x$ is given by $x=[0;\varepsilon_1/a_1, \varepsilon_2/a_2,\dots, \varepsilon_n/a_n, \varepsilon_{n+1}/a_{n+1},\dots]$, then one can show that
$$
t_n = T_{\alpha}^n(x) = [0; \varepsilon_{n+1}/a_{n+1},\varepsilon_{n+2}/a_{n+2},\dots ]
$$
and using induction and the recurrence relations~(\ref{eq:pnrecurrencerelation}) and~(\ref{eq:qnrecurrencerelation}),
$$
v_n = [0; 1/a_n, \varepsilon_n/a_{n-1},\varepsilon_{n-1}/a_{n-2},\dots , \varepsilon_2/a_1] = \frac{q_{n-1}}{q_n}.
$$
So $t_n$ is the \emph{future} of $x$ at time $n$, and $v_n$ is the \emph{past} of $x$ at time $n$ in the $\alpha$-expansion of $x$. Of course, for different $\alpha$ the future and past of $x$ at time $n$ might be different. For details and proofs, please consult~\cite{[DK],[IK]}.

\subsection{Singularizations and insertions, part 1}\label{subsec:singularizationinsertions1}
In this section the following identities will be frequently used:
\begin{itemize}
\item[ ]
\textbf{Singularization.} Let $A, B\in \Z$ with $A\geq 0$ and $B\geq 1$  and let $\xi> -1$, then:
$$
A + \frac{1}{1 + \displaystyle{\frac{1}{B+\xi}}} = A+1 +\frac{-1}{B+1+\xi} .
$$

\item[ ]

\item[ ]
\textbf{Insertion.}  Let $A, B\in \N$, $B\geq 3$, and let $\xi> -1$, then:
$$
A+\frac{-1}{B+\xi} = A + 1 + \frac{-1}{1 + \displaystyle{\frac{-1}{B+1+\xi}}}.
$$
\end{itemize}
The proofs of these two identities are left to the reader. In Subsection~\ref{subsec:grotesquesystem} and Section~\ref{subsec:smalleralpha} we will see variations of these singularization and insertion identities.\medskip\

Now let $\alpha\in [g,1)$. In the remainder of this section we will show how we can obtain $\Omega_{\alpha}$ from $\Omega_1$, using singularizations and insertions. Our construction will show that all these systems are (metrically) isomorphic, and therefore have -- among other things -- the same entropy. Now let $x\in I_1=[-1,1)$, and suppose that $n\geq 0$ is the first ``time'' that $T_1^n(x)\geq \alpha$ (note that $n=0$ is allowed; $n$ is the first ``time'' that $T_1^n(x)\not\in I_{\alpha} = [\alpha -1,\alpha )$). Since $g\leq \alpha <1$ and since
the time-1 cylinder $\Delta (+1,1)$ for the $T_1$-map is given by
$$
\Delta (+1,1) = \{ x\in I_1\, |\, \varepsilon (x) = +1,\, d_1(x) = 1 \} = \left( \tfrac{1}{2},1\right),
$$
the $T_1$-expansion of $x$ is given by:
\begin{equation}\label{eq:T1expansionofx}
x = \frac{\varepsilon_1}{a_1+\ddots +\displaystyle{\frac{\varepsilon_n}{a_n+ \displaystyle{\frac{1}{1+\displaystyle{\frac{1}{a_{n+2}+\displaystyle{\frac{\varepsilon_{n+3}}{a_{n+3}+\ddots}}}}}}}}} .
\end{equation}
Here $\varepsilon_n=\varepsilon (T_1^{n-1}(x))$ and $a_n=d_1(T_1^{n-1}(x))$, for $n\geq 1$ (note that the expansion of $x$ in~(\ref{eq:T1expansionofx}) is finite if and only if $x\in\Q$). So if $T_1^n(x)\geq \alpha$ we have that $a_{n+1}=1=\varepsilon_{n+1}=\varepsilon_{n+2}$. The expansion from~(\ref{eq:T1expansionofx}) is denoted as
$$
x = [0;\, \varepsilon_1/a_1, \dots, \varepsilon_n/a_n, 1/1, 1/a_{n+2}, \varepsilon_{n+3}/a_{n+3},\dots ].
$$
Now singularizing $a_{n+1}=1$ in~(\ref{eq:T1expansionofx}) yields the following continued fraction expansion of $x$:
\begin{equation}\label{eq:T1expansionofxintemediate}
x = \frac{\varepsilon_1}{a_1+\ddots +\displaystyle{\frac{\varepsilon_n}{a_n+1+ \displaystyle{\frac{-1}{a_{n+2}+1+\displaystyle{\frac{\varepsilon_{n+3}}{a_{n+3}+\ddots}}}}}}} .
\end{equation}
In~(\ref{eq:T1expansionofxintemediate}) two of the partial quotients are \emph{even} (viz.\ $a_n+1$ and $a_{n+2}+1$), so this expansion of $x$ is \textbf{not} a continued fraction expansion of $x$ with (only!) odd partial quotients. Now inserting in~(\ref{eq:T1expansionofxintemediate}) $-1/1$ ``between'' $a_n+1$ and $\frac{-1}{a_{n+2}+1+\dots}$ yields the following continued fraction expansion of $x$ with (only) odd partial quotients:
\begin{equation}\label{eq:T1expansionofxAFTERintemediate}
x = \frac{\varepsilon_1}{a_1+\ddots +\displaystyle{\frac{\varepsilon_n}{a_n+2+\displaystyle{\frac{-1}{1 +  \displaystyle{\frac{-1}{a_{n+2}+2+\displaystyle{\frac{\varepsilon_{n+3}}{a_{n+3}+\ddots}}}}}}}}} ;
\end{equation}
which we denote by $x=[a_0^*; \varepsilon_1^*/a_1^*,\varepsilon_2^*/a_2^*,\dots]$; here $a_0^*=-1$ if $x\in[\alpha,1)$, otherwise it is $0$. Due to our construction we have that $\varepsilon_1^*=\varepsilon_1,\dots,\varepsilon_n^*=\varepsilon_n$, $\varepsilon_{n+1}^* = -1=\varepsilon_{n+2}^*$, and that $\varepsilon_{n+k}^* = \varepsilon_{n+k}$, for $k\geq3$. Similarly, for the partial quotients we have that
$a_1^*=a_1,\dots,a_{n-1}^*=a_{n-1}$, $a_n^*=a_n+2$, $a_{n+1}^*=1$, $a_{n+2}^*=a_{n+2}+2$, and that $a_{n+k}^*=a_{n+k}$, for $k\geq 3$. Setting
$$
t_i^* = \frac{\varepsilon_{i+1}^*}{a_{i+1}^*+\ddots} = [0; \varepsilon_{i+1}^*/a_{i+1}^*,\dots],
$$
and
$$
v_i^* = \frac{1}{a_i^*+\displaystyle{\frac{\varepsilon_i^*}{a_{i-1}^*+ \ddots + \displaystyle\frac{\varepsilon_2^*}{a_1^*}}}} = [0;1/a_i^*,\varepsilon_i^*/a_{i-1}^*,\dots,\varepsilon_2^*/a_1^*],
$$
we see that
$$
(t_i^*,v_i^*) = (t_i,v_i)\quad \text{for $i=1,2,\dots,n-1$ and $i\geq n+2$},
$$
and that $(t_n,v_n)$ and $(t_{n+1},v_{n+1})$ got replaced by $(t_n^*,v_n^*)$ respectively $(t_{n+1}^*,v_{n+1}^*)$.\smallskip\

Since $(t_n,v_n)\in D_{\alpha} = [\alpha ,1)\times [0,G]$ we immediately have that
$$
(t_{n+1},v_{n+1})\in H_{\alpha} := \mathcal{T}_1(D_{\alpha }) = \Big( 0, \frac{1-\alpha}{\alpha}\Big] \times \Big[ g^2,1\Big].
$$
For the new continued fraction expansion~(\ref{eq:T1expansionofxAFTERintemediate}) of $x$ these rectangles $D_{\alpha}$ and $H_{\alpha}$ have been ``vacated''; see also Figure~\ref{fig:Omega1Omegaalpha}.\medskip\

Where do $(t_n^*,v_n^*)$ and $(t_{n+1}^*,v_{n+1}^*)$ ``live''? We see from~(\ref{eq:T1expansionofx}) that the time $n$ future $t_n$ in the $\alpha = 1$-expansion of $x$ is given by
$$
t_n = \frac{1}{1+\displaystyle{\frac{1}{a_{n+2}+\displaystyle{\frac{\varepsilon_{n+3}}{a_{n+3} + \dots }}}}},
$$
while for the ``new'' time $n$ future $t_n^*$ of $x$ in~(\ref{eq:T1expansionofxAFTERintemediate}) we have:
$$
t_n^* = \frac{-1}{1+\displaystyle{\frac{-1}{a_{n+2}+2+\dots}}}.
$$
Since
\begin{eqnarray*}
t_n &=& \frac{1}{1+\displaystyle{\frac{1}{a_{n+2}+ \dots}}} \,\, =\,\, 1 + \frac{-1}{a_{n+2}+1+\dots} \\
&=& 1+1+ \frac{-1}{1+\displaystyle{\frac{-1}{a_{n+2}+1+1+\dots}}} \,\, = \,\, 2+t_n^*,
\end{eqnarray*}
we see that
\begin{equation}\label{eq:tnVStn*}
t_n^* = t_n-2.
\end{equation}
Note that this is something we could expect beforehand from~(\ref{eq:definitionmapattimen}); since $t_n\in [\alpha ,1)$ we must increase the digit $a_n$ by 2 to get that $t_n^*\in I_{\alpha}$.\medskip\

For the time $n$ past $v_n$ in the $\alpha = 1$-expansion of $x$ we have that
$$
v_n = \frac{q_{n-1}}{q_n} = [0;1/a_n, \varepsilon_n/a_{n-1}, \dots , \varepsilon_2/a_1],
$$
while we see from~(\ref{eq:T1expansionofxAFTERintemediate}), and also~(\ref{eq:qnrecurrencerelation}) that the ``new'' past $v_n^*$ satisfies
\begin{eqnarray*}
v_n^* &=& [0; 1/a_n+2,\varepsilon_n/a_{n-1},\dots , \varepsilon_2/a_1] \,\, =\,\, \frac{1}{a_n+2+\varepsilon_n v_{n-1}} \\
&=& \frac{1}{a_n+2 +\varepsilon_n \displaystyle{\frac{q_{n-2}}{q_{n-1}}}} \,\, =\,\, \frac{q_{n-1}}{2q_{n-1} + a_nq_{n-1}+\varepsilon_nq_{n-2}}\\
&=& \frac{q_{n-1}}{2q_{n-1}+q_n} \,\, =\,\, \frac{\frac{q_{n-1}}{q_n}}{2\frac{q_{n-1}}{q_n}+1} \,\, =\,\, \frac{v_n}{2v_n+1}.
\end{eqnarray*}
In view of this and~(\ref{eq:tnVStn*}) we define a map $M_{\alpha}:\Omega_1\to\R^2$ as:
\begin{equation}\label{eq:MalphaMap1}
M_{\alpha}(t,v) = \begin{cases}
(t,v), & \text{if $(t,v)\in\Omega_1\setminus D_{\alpha}$}, \\
 & \\
\left( t-2,\displaystyle{\frac{v}{1+2v}}\right) , & \text{if $(t,v)\in D_{\alpha}$}.
\end{cases}
\end{equation}
Since $\alpha \in [g,1)$, we have that $\frac{1}{\alpha}-1\leq \alpha$, and that equality holds when $\alpha = g$. For $\alpha < g$ something special happens; see Section~\ref{subsec:smalleralpha}.

So we found that
$$
(t_n,v_n)\in D_{\alpha} \quad \text{if and only if} \quad (t_n^*,v_n^*)\in M_{\alpha}(D_{\alpha}) = [\alpha -2,-1)\times \Big[ 0, g^2\Big] .
$$
A calculation yields that the probability measure $\bar{\mu}_1$ on $\Omega_1$ is invariant under the map $M_{\alpha}$.\medskip\

Next we will determine where $(t_{n+1}^*,v_{n+1}^*)$ lives. From~(\ref{eq:T1expansionofxAFTERintemediate}) we see that
$$
t_n^* = \frac{-1}{1+t_{n+1}^*},
$$
implying that
$$
t_{n+1}^* = \frac{-1}{t_n^*} -1.
$$
Now for $t\in [\alpha -2, -1)$ we have that
\begin{equation}\label{eq:tnstar}
T_{\alpha}(t) = \frac{-1}{t}-1,
\end{equation}
so we see that $t_{n+1}^* = T_{\alpha}(t_n^*)$. Furthermore, from~(\ref{eq:T1expansionofxAFTERintemediate}) we see that
$$
v_{n+1}^* = [0; 1/1, -1/a_n+2,\varepsilon_n/a_{n-1},\dots,\varepsilon_2/a_1]
$$
and that
$$
v_n^* = [0; 1/a_n+2,\varepsilon_n/a_{n-1},\dots,\varepsilon_2/a_1],
$$
from which we see that
\begin{equation}\label{eq:vnstar}
v_{n+1}^* = \frac{1}{1-v_n^*}.
\end{equation}
From~(\ref{eq:tnstar}) and~(\ref{eq:vnstar}) we thus find that
$$
(t_{n+1}^*,v_{n+1}^*) = \mathcal{T}_{\alpha}(t_n^*,v_n^*).
$$
Defining $M_{\alpha}: [\alpha -2,-1)\times [0,g^2]\to \R^2$ by
$$
M_{\alpha}(t,v) = \left( \frac{-1}{t}-1,\frac{1}{1-v}\right) ( = \mathcal{T}_{\alpha}((t,v))),
$$
and setting
$$
\Omega_{\alpha} = (\Omega_1\setminus (D_{\alpha}\cup \mathcal{T}_1(D_{\alpha})) \cup M_{\alpha}(D_{\alpha})\cup M_{\alpha}^2(D_{\alpha}),
$$
(cf.~Figure~\ref{fig:Omega1Omegaalpha}), it follows from the fact that for $k\geq 2$ we have that $(t_{n+k}^*,v_{n+k}^*) = (t_{n+k},v_{n+k})$ we must have that
$$
\mathcal{T}_{\alpha}\left( M_{\alpha}^2(D_{\alpha})\right) = \mathcal{T}_{\alpha}\left( H_{\alpha}\right);
$$
the ``new material'' $M_{\alpha}^2(D_{\alpha})$ fills the hole in $\Omega_{\alpha}$ caused by the ``hole'' $H_{\alpha}$. The map $\mathcal{T}_{\alpha}$ is bijective on $\Omega_{\alpha}$, apart from a set of Lebesgue measure 0. We find that the dynamical system
$$
(\Omega_{\alpha},\mathcal{B}_{\alpha},\bar{\mu}_{\alpha}, \mathcal{T}_{\alpha})
$$
is metrically isomorphic to the dynamical system
$$
(\Omega_1,\mathcal{B}_1,\bar{\mu}_1, \mathcal{T}_1)
$$
and therefore inherits all its ergodic properties. In particular, since the map $M_{\alpha}$ preserves the $\bar{\mu}_1$-meaure, we find that on $\Omega_{\alpha}$ the $\mathcal{T}_{\alpha}$-invariant measure $\bar{\mu}_{\alpha}$ has density
$$
\frac{1}{3\log G}\, \frac{1}{(1+xy)^{2}}
$$
on $\Omega_{\alpha}$. Furthermore, since $(\Omega_1,\mathcal{B}_1,\bar{\mu}_1, \mathcal{T}_1)$ is (a version of) the natural extension of $T_1$ (cf.~\cite{[S1],[S2],[Rie]}), we find that $(\Omega_{\alpha},\mathcal{B}_{\alpha},\bar{\mu}_{\alpha}, \mathcal{T}_{\alpha})$ is the natural extension of $T_{\alpha}$ for $g\leq \alpha <1$; this is the system obtained by Boca and the fourth author in~\cite{[BM]}.
\begin{figure}[ht]
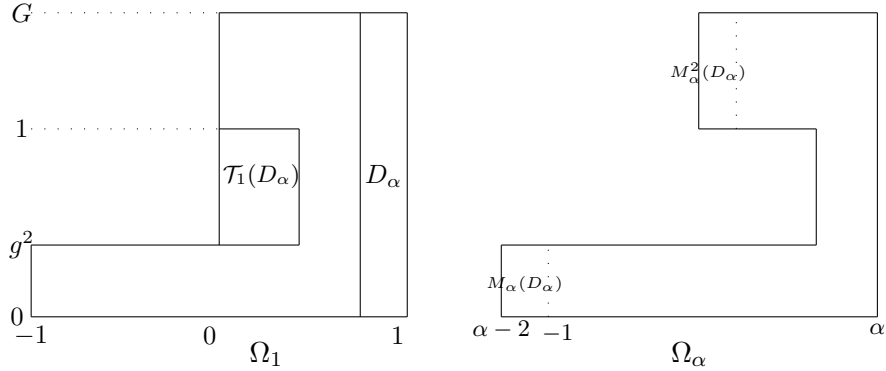

\beginpicture

\setcoordinatesystem units <0.125cm,0.125cm> 
\setplotarea x from -10 to 60, y from -10 to 40

\putrule from 0 0 to 40 0
\putrule from 40 0 to 40 32.36068
\putrule from 20 32.36068 to 40 32.36068
\putrule from 20 7.63932 to 20 32.36068
\putrule from 0 7.63932 to 28.5 7.63932
\putrule from 0 0 to 0 7.63932

\putrule from 35 0 to 35 32.36068

\putrule from 20 20 to 28.5 20
\putrule from 28.5 7.63932 to 28.5 20

\put {$g^2$} at -1 7.7
\put {$-1$} at 0 -2
\put {$0$} at 19 -2
\put {$1$} at 39 -2
\put {\large $\Omega_1$} at 25 -4
\put {$0$} at -1.5 0
\put {$G$} at -1  32.36068
\put {$1$} at -1 20

\put {$D_\alpha$} at 37.5 15
\put {\small $\mathcal{T}_1(D_\alpha)$} at  24.5 15

\putrule from 50 0 to 90 0
\putrule from 71 32.36068 to 90 32.36068
\putrule from 50 7.63932 to 83.5 7.63932
\putrule from 50 0 to 50 7.63932

\putrule from 90 0 to 90 32.36068
\putrule from 71 20 to 71 32.36068
\putrule from 71 20 to 83.5 20
\putrule from 83.5 7.63932 to 83.5 20

\put {\small $\alpha-2$} at 50 -1.3
\put {\small $-1$} at 56 -1.75
\put {\small $\alpha$} at 90 -1.3

\put {\tiny $M_\alpha(D_\alpha)$} at 52.5 3.5
\put {\tiny $M_\alpha^2(D_\alpha)$} at 72 26

\put {\large $\Omega_\alpha$} at 70 -4


\setdots
\putrule from 0 20 to 20 20
\putrule from 0 32.36068 to 20 32.36068
\putrule from 75 20 to 75 32.36068
\putrule from 55 0 to 55 7.63932
\endpicture
\caption{The domains $\Omega_1$ and $\Omega_\alpha$ for $\alpha\in[g,1)$.}
\label{fig:Omega1Omegaalpha}
\end{figure}

\subsection{Singularizations and insertions, part 2}\label{subsec:grotesquesystem}
In order to obtain the results from~\cite{[BM]} for $\alpha \in (1,G]$ we could again start from the Schweiger-Rieger natural extension, and use appropriate singularizations and insertions to find the systems obtained in~\cite{[BM]} for these values of $\alpha$. In view of what we will see in Section~\ref{subsec:smalleralpha} we decided to start with the natural extension of the \emph{grotesque continued fraction expansion}. This dynamical system, which is case $\alpha = G$ in~\cite{[BM]} is already present in the papers by Schweiger (cf.~\cite{[S1],[S2]}) and Rieger (\cite{[Rie]}). So our starting point is the planar natural extension $\Omega_G = [-g^2,G)\times [0,1]$; see Figure~\ref{fig:OmegaG}.

\begin{figure}[ht]
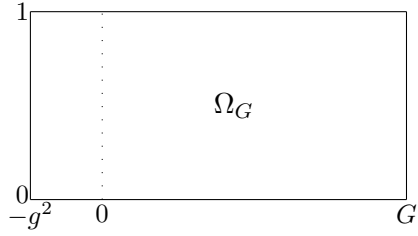

\beginpicture 

\setcoordinatesystem units <0.125cm,0.125cm> 
\setplotarea x from -35 to 30, y from -5 to 25

\putrule from -7.63932 0 to 32.36068 0
\putrule from 32.36068 0 to 32.36068 20
\putrule from -7.63932 20 to 32.36068 20
\putrule from -7.63932 0 to -7.63932 20

\put {$0$} at -8.5 0.5
\put {$1$} at -8.5 20
\put {$-g^2$} at -7.63932 -1.7
\put {$0$} at 0 -1.5
\put {$G$} at 32.36068 -1.5
\put {\large $\Omega_G$} at 14 10

\setdots
\putrule from 0 0 to 0 20

\endpicture
\caption{The domain $\Omega_G$.}
\label{fig:OmegaG}
\end{figure}

In this subsection we will frequently use  singularizations and insertions which are slightly different from the singularization defined in the previous subsection:
\begin{itemize}
\item[ ]

\textbf{Singularization}. Let $A,B\in\Z$, $A\geq 0$, $B\geq 1$ and $\xi > -1$, then:
$$
A + \frac{1}{1+ {\displaystyle \frac{-1}{B+\xi}}} = A+1 + \frac{1}{B-1+\xi}.
$$
\item[ ]

\item[ ]
\textbf{Insertion}. Let $A,B\in\Z$, $A\geq 0$, $B\geq 3$ and $\xi> -1$, then:
$$
A+\frac{1}{B+\xi} = A+1 + \frac{-1}{1 + {\displaystyle {\frac{1}{B-1+\xi}} }}.
$$
\end{itemize}

Now let $1<\alpha <G$, then we want to find the planar natural extension $\Omega_{\alpha}$. Let $x\in I_G$ and let $n\geq 0$ be the first ``time'' we have that $T_G^n(x)\geq \alpha$. So if we again define for $k\geq 0$: $(t_k,v_k)=\mathcal{T}_G^k(x,0)$, then $n$ is the first non-negative integer for which $(t_n,v_n)\in D_{\alpha} = [\alpha ,G)\times [0,1]$. Since $T_G(1) = 0$, we find that $\varepsilon_{n+2} = \text{sign}(T_G^{n+1}(x)) = -1$ and that $a_{n+2}$ is at least equal to 3 (viz.\ $T_G([1,G]) = [-g^2,0]$ and $1/g^2 = G^2 = G+1$). But then the $G$-expansion of $x$ is given by
\begin{equation}\label{eq:GexpansionOfx}
x = [0; \varepsilon_1/a_1, \dots, \varepsilon_n/a_n, 1/1, -1/a_{n+2},\varepsilon_{n+3}/a_{n+3},\dots ].
\end{equation}
As in the previous subsection we want to ``skip'' $D_{\alpha}$; $t_n$ is ``too big''. To achieve this, we singularize $a_{n+1}=1$, to find the following continued fraction expansion of $x$:
\begin{equation}\label{eq:intemediateexpansionOfx}
x = [0; \varepsilon_1/a_1, \dots, \varepsilon_n/a_n+1, 1/a_{n+2}-1,\varepsilon_{n+3}/a_{n+3},\dots ].
\end{equation}
Now both $a_n+1$ and $a_{n+2}-1$ are even, so~(\ref{eq:intemediateexpansionOfx}) is \textbf{not} an odd continued fraction expansion of $x$. We need an insertion of $-1/1$ to arrive at
\begin{equation}\label{eq:targetexpansionOfx}
x = [0; \varepsilon_1/a_1, \dots, \varepsilon_n/a_n+2,-1/1, 1/a_{n+2}-2,\varepsilon_{n+3}/a_{n+3},\dots ].
\end{equation}
If we denote this last expansion~(\ref{eq:targetexpansionOfx}) of $x$ as $x=[a_0^*; \varepsilon_1^*/a_1^*, \dots]$, then we have that $a_0^* = 1$ if $n=0$, and $a_0^*=0$ otherwise, and that $\varepsilon_1^*=\varepsilon_1,\dots,\varepsilon_n^*=\varepsilon_n$, $\varepsilon_{n+1}^* = -1$, $\varepsilon_{n+2}^*=1$, and that $\varepsilon_{n+k}^* = \varepsilon_{n+k}$, for $k\geq3$. Similarly, for the partial quotients we have that
$a_1^*=a_1,\dots,a_{n-1}^*=a_{n-1}$, $a_n^*=a_n+2$, $a_{n+1}^*=1$, $a_{n+2}^*=a_{n+2}-2$, and that $a_{n+k}^*=a_{n+k}$, for $k\geq 3$. As in the previous subsection we can find the relation between  $(t_n,v_n)$ and $(t_n^*,v_n^*)$, and $(t_{n+1},v_{n+1})$ and $(t_{n+1}^*,v_{n+1}^*)$; we skip the details, as these are similar to what we did in the previous subsection, and essentially applying the above steps of first a singularization and then a singularization to $t_n$ and $v_n$. As in the previous subsection we have that
$$
(t_n^*,v_n^*) = \left( t_n-2,\frac{v_n}{1+2v_n}\right) .
$$
In view of this we define for $1< \alpha <G$ a map $M_{\alpha}:\Omega_G\to\R^2$ similar to the definition of $M_{\alpha}$ from~(\ref{eq:MalphaMap1}):
\begin{equation}\label{eq:MalphaMapG}
M_{\alpha}(t,v) = \begin{cases}
(t,v), & \text{if $(t,v)\in\Omega_G\setminus D_{\alpha}$}, \\
 & \\
\left( t-2,\displaystyle{\frac{v}{1+2v}}\right) , & \text{if $(t,v)\in D_{\alpha}$}.
\end{cases}
\end{equation}
Since $G-2 = -g^2$, we that $M_{\alpha}(D_{\alpha}) = [\alpha -2,-g^2) \times [0,\tfrac{1}{3}$]; see also Figure~\ref{fig:OmegaGomegaalpha1}.

\begin{figure}[ht]
\beginpicture 

\setcoordinatesystem units <0.125cm,0.125cm> 
\setplotarea x from -10 to 60, y from -10 to 40

\putrule from -7.63932 0 to 32.36068 0
\putrule from 32.36068 0 to 32.36068 20
\putrule from -7.63932 20 to 32.36068 20
\putrule from -7.63932 0 to -7.63932 20

\putrule from 26 0 to 26 20

\putrule from -7.63932 10 to -4.615385 10
\putrule from -4.615385 10 to -4.615385 20

\put {$0$} at -8.5 0.5
\put {$\frac{1}{2}$} at -8.5 10
\put {$1$} at -8.5 20

\put {$-g^2$} at -7.9 -1.8
\put {$\frac{1-\alpha}{\alpha}$} at -4.615385 -5
\put {$0$} at 0 -1.5
\put {$\alpha$} at 26 -1.5
\put {$G$} at 32.36068 -1.5
\put {$D_\alpha$} at 29 10
\put {\tiny $H_{\alpha}$} at -5.9 15
\put {\large $\Omega_G$} at 14 -4.5

\putrule from 44.36068 0 to 82 0
\putrule from 88.36068 20 to 88.36068 30
\putrule from 51.384615 20 to 64.571429 20
\putrule from 82 20 to 88.36068 20
\putrule from 48.36068 6.66666 to 48.36068 10

\putrule from 44.36068 6.66666 to 48.36068 6.66666
\putrule from 44.36068 0 to 44.36068 6.66666

\putrule from 64.571429 30 to 88.36068 30
\putrule from 64.571429 20 to 64.571429 30

\putrule from 82 0 to 82 20

\putrule from 48.36068 10 to 51.384615 10
\putrule from 51.384615 10 to 51.384615 20

\put {$0$} at 43 0.5
\put {$\frac{1}{3}$} at 43 6
\put {$\frac{1}{2}$} at 43 10.2
\put {$1$} at 43 20
\put {$\frac{3}{2}$} at 43 30

\put {$-g^2$} at 48 -1.8
\put {$\frac{1-\alpha}{\alpha}$} at 51.384615 -5
\put {\tiny $T_\alpha(\alpha-2)$} at 64.571429 -1.5
\put {$\alpha$} at 82 -1.5
\put {$G$} at 88.36068 -1.5
\put {\large $\Omega_{\alpha,1}$} at 70 -4.5
\put {\tiny $M_\alpha(D_{\alpha})$} at 44 3
\put {\tiny $M^2_\alpha(D_{\alpha})$} at 76 25

\setdots
\putrule from 0 0 to 0 20
\putrule from -4.615385 -3 to -4.615385 10

\putrule from 48.36068 0 to 48.36068 6.66666
\putrule from 51.384615 -3 to 51.384615 10
\putrule from 64.571429 0 to 64.571429 20
\putrule from  88.36068 0 to  88.36068 20

\putrule from 64.571429 20 to 82 20

\putrule from 44.36068 10 to 51.384615 10
\putrule from 44.36068 20 to 51.384615 20
\putrule from 44.36068 30 to 64.571429 30

\endpicture
\caption{The domain $\Omega_G$ and the region $\Omega_{\alpha,1}$.}
\label{fig:OmegaGomegaalpha1}
\end{figure}

As in the previous subsection we find that
$$
(t_{n+1}^*,v_{n+1}^*) = \mathcal{T}_{\alpha}(t_n^*,v_n^*),
$$
and therefore we extend the definition of $M_{\alpha}$ to $M_{\alpha}(D_{\alpha})$ by
$$
M_{\alpha}(t,v) = \left( \frac{-1}{t}-1,\frac{1}{1-v}\right) ( = \mathcal{T}_{\alpha}((t,v))),
$$
Note that $M_{\alpha}$ preserves the invariant measure with density $(3\log G)^{-1}(1+xy)^{-2}$ both on $D_{\alpha}$ and on $M_{\alpha}(D_{\alpha})$. We have that
$$
M_{\alpha}^2(D_{\alpha}) = [T_{\alpha}(\alpha -2), G)\times [1,\tfrac{3}{2}] .
$$
Removing $D_{\alpha}$ and $H_{\alpha}=\mathcal{T}_G(D_{\alpha}) = [-g^2,\tfrac{1-\alpha}{\alpha})\times [\tfrac{1}{2},1]$ from $\Omega_G$ we find as a new planar domain,
$$
\Omega_{\alpha ,1} = \left( \Omega_G\setminus (D_{\alpha}\cup H_{\alpha})\right) \cup M_{\alpha}(D_{\alpha})\cup M_{\alpha}^2(D_{\alpha}).
$$
One might be inclined to think that $\Omega_{\alpha ,1}$ is the planar natural extension of $T_{\alpha}$ $\dots$ but this is not correct. 
Note that $\Omega_{\alpha ,1}$ contains the rectangle
$D_{\alpha ,1} = [\alpha ,G)\times [1,\tfrac{3}{2})$. Extending the definition of $M_{\alpha}$ in an obvious way to $D_{\alpha ,1}$ and $M_{\alpha}(D_{\alpha,1})$ we see that
$$
M_{\alpha}(D_{\alpha ,1}) = [\alpha -2, -g^2)\times [\tfrac{1}{3},\tfrac{3}{8}]\quad \text{and}\quad
M_{\alpha}^2(D_{\alpha ,1}) = [T_{\alpha}(\alpha -2), G)\times [\tfrac{3}{2},\tfrac{8}{5}];
$$

Removing $D_{\alpha,1}$ and $H_{\alpha,2}=\mathcal{T}_G(D_{\alpha,1}) = [-g^2,\tfrac{1-\alpha}{\alpha})\times [\tfrac{2}{5},\tfrac{1}{2}]$ from $\Omega_{\alpha,1}$ we find as a new planar domain
$$
\Omega_{\alpha ,2} = \left( \Omega_{\alpha,1}\setminus (D_{\alpha,1}\cup H_{\alpha,1})\right) \cup M_{\alpha}(D_{\alpha,1})\cup M_{\alpha}^2(D_{\alpha,1}),
$$
see Figure~\ref{fig:Omegaalpha1omegaalpha2}.
\begin{figure}[ht]
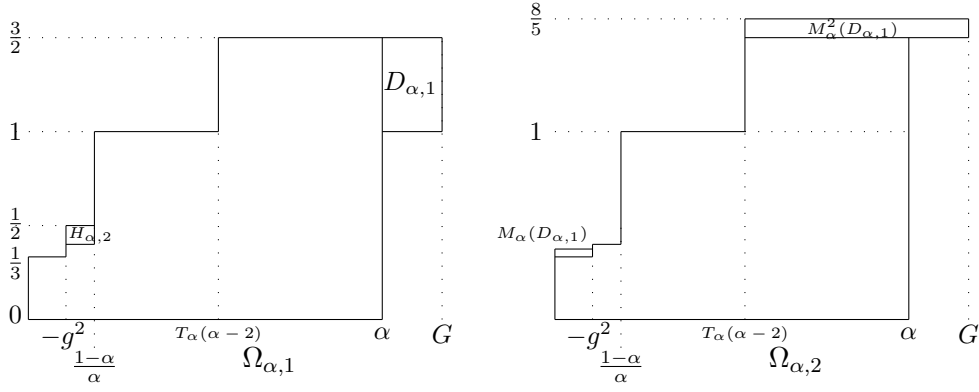

\beginpicture 

\setcoordinatesystem units <0.125cm,0.125cm> 
\setplotarea x from -10 to 60, y from -10 to 40

\putrule from -11.63932 0 to 26 0
\putrule from 32.36068 20 to 32.36068 30
\putrule from -4.615385 20 to 8.571429 20
\putrule from 26 20 to 32.36068 20
\putrule from -7.63932 6.66666 to -7.63932 10

\putrule from -11.63932 6.66666 to -7.63932 6.66666
\putrule from -11.63932 0 to -11.63932 6.66666

\putrule from 8.571429 30 to 32.36068 30
\putrule from 8.571429 20 to 8.571429 30

\putrule from 26 0 to 26 30

\putrule from -7.63932 10 to -4.615385 10
\putrule from -4.615385 8 to -4.615385 20
\putrule from -7.63932 8 to -4.615385 8

\put {\tiny $H_{\alpha,2}$} at -5 9
\put {$D_{\alpha,1}$} at 29 25

\put {$0$} at -13 0.5
\put {$\frac{1}{3}$} at -13 6
\put {$\frac{1}{2}$} at -13 10.2
\put {$1$} at -13 20
\put {$\frac{3}{2}$} at -13 30

\put {$-g^2$} at -7.9 -1.7
\put {$\frac{1-\alpha}{\alpha}$} at -4.615385 -5
\put {\tiny $T_\alpha(\alpha-2)$} at 8.571429 -1.5
\put {$\alpha$} at 26 -1.5
\put {$G$} at 32.36068 -1.5
\put {\large $\Omega_{\alpha,1}$} at 14 -4.5

\putrule from 44.36068 0 to 82 0
\putrule from 51.384615 20 to 64.571429 20

\putrule from 48.36068 6.66666 to 48.36068 8

\putrule from 44.36068 6.66666 to 48.36068 6.66666
\putrule from 44.36068 0 to 44.36068 7.5
\putrule from 44.36068 7.5 to 48.36068 7.5

\putrule from 64.571429 30 to 88.36068 30
\putrule from 64.571429 20 to 64.571429 32

\putrule from 82 0 to 82 30 

\putrule from 48.36068 8 to 51.384615 8
\putrule from 51.384615 8 to 51.384615 20

\putrule from 88.36068 30 to 88.36068 32
\putrule from 64.571429 32 to 88.36068 32

\put {$-g^2$} at 48 -1.7
\put {$\frac{1-\alpha}{\alpha}$} at 51.384615 -5
\put {\tiny $T_\alpha(\alpha-2)$} at 64.571429 -1.5
\put {$\alpha$} at 82 -1.5
\put {$G$} at 88.36068 -1.5
\put {\large $\Omega_{\alpha,2}$} at 70 -4.5

\put {\tiny $M_\alpha(D_{\alpha,1})$} at 43 9
\put {\tiny $M^2_\alpha(D_{\alpha,1})$} at 76 30.9

\put {$1$} at 42.36068 20
\put {$\frac{8}{5}$} at 42.36068 32

\setdots
\putrule from 48.36068 0 to 48.36068 6.66666
\putrule from 51.384615 -3 to 51.384615 10
\putrule from 64.571429 0 to 64.571429 20
\putrule from  88.36068 0 to  88.36068 30

\putrule from -7.63932 0 to -7.63932 6.66666
\putrule from -4.615385 -3 to -4.615385 10
\putrule from 8.571429 0 to 8.571429 20
\putrule from  32.36068 0 to  32.36068 30

\putrule from 64.571429 20 to 82 20

\putrule from 44.36068 20 to 51.384615 20
\putrule from 44.36068 32 to 64.571429 32

\putrule from -13 10 to -7.63932 10
\putrule from -13 20 to -4.615385 20
\putrule from -13 30 to 8.571429 30

\endpicture
\caption{The regions $\Omega_{\alpha,1}$ and $\Omega_{\alpha,2}$.}
\label{fig:Omegaalpha1omegaalpha2}
\end{figure}

We find that $\Omega_{\alpha,2}$ is also not the planar natural extension of $T_{\alpha}$; the domain $\Omega_{\alpha ,2}$ contains the rectangle
$D_{\alpha ,2} = [\alpha ,G)\times [\tfrac{3}{2},\frac{8}{5})$, which has a first coordinate which is ``too large''. If we now repeat the above procedure we find a sequence of rectangles $D_{\alpha,2}, D_{\alpha,3}, \dots,D_{\alpha,n},\dots$ such that if $D_{\alpha, 0}=D_{\alpha}=[\alpha , G) \times [0,1]$, and for $n\geq 0$,
$$
D_{\alpha,n} = [\alpha , G) \times [c_{2n-2},c_{2n}],
$$
where $(c_n)_{n\geq 0}$ is the sequence of regular continued fraction expansion convergents of $G$.

To see why this is indeed the case, we define the map $M_{\alpha}:  [ \alpha ,G)\times [0,G] \cup [ \alpha -2,-g^2)\times [0,g) \to \R^2$ as
$$
M_{\alpha}(t,v) = \begin{cases}
\left( t-2, \frac{v}{1+2v}\right), & \text{if $(t,v)\in [ \alpha,G)\times [0,G]$}; \\
 & \\
\left( \frac{-1}{t}-1, \frac{1}{1-v}\right) , & \text{if $(t,v)\in [ \alpha -2,-g^2)\times [0,g]$}.
\end{cases}
$$
So for $t=G$ and $0\leq v\leq G$ we find that
$$
M_{\alpha}^2(t,v) = \left( G, m(v))\right) ,
$$
where the second coordinate map $m(v)$ satisfies $m(0)=1$ and for $0<v\leq G$:
$$
m(v) = \frac{1}{1-\displaystyle{\frac{v}{1+2v}}} = \frac{2v+1}{v+1} = 1 + \frac{v}{1+v} = 1+\frac{1}{1+\displaystyle{\frac{1}{v}}}.
$$
But then we have for $0<v\leq G$:
$$
m^2(v) = 1+\frac{1}{1+\displaystyle{\frac{1}{\displaystyle{1+\frac{1}{1+\displaystyle{\frac{1}{v}}}}}}}.
$$
Setting $m_0=0$, $m_1=m(0)=1$, $m_2=m^2(0)=m(1)=\tfrac{3}{2}$, $\dots$, $m_n=m^n(0)=m^{n-1}(1)$, for $n\geq 1$, we see that
\begin{equation}\label{eq:c_n}
m_n = [1; \underbrace{1,1,\dots,1}_{2(n-1)\,\,\text{times}}] \uparrow G,\quad \text{if $n\to\infty$}.
\end{equation}
We denote the regular continued fraction convergent of $G$ in~(\ref{eq:c_n}) by $[1; 1^{2n-2}]$. Setting
$$
D_{\alpha,n} = [\alpha,G)\times [m_n,m_{n+1}],\quad \text{for $n\geq 0$}
$$
then by construction we have that
$$
M_{\alpha}^2(D_{\alpha,n}) =  [T_{\alpha}(\alpha -2),G)\times [m_{n+1},m_{n+2}],\quad \text{for $n\geq 0$},
$$
which contains $D_{\alpha,n+1}$ since $T_{\alpha}(\alpha -2) \leq \alpha$. Furthermore, we see that at each time-step $n$ we add as a ``new'' region:
$$
M_{\alpha}(D_{\alpha,n}) = [\alpha -2,-g^2)\times \left[ \frac{m_n}{1+2m_n},\frac{m_{n+1}}{1+2m_{n+1}}\right].
$$
Note that for $0\leq v\leq G$ we have that:
$$
\frac{v}{1+2v} = \frac{1}{2+\displaystyle{\frac{1}{v}}},
$$
and therefore for $n\geq 1$ we have that
$$
\frac{m_n}{1+2m_n} = [0;2, 1^{2n-1}] \uparrow g^2, \quad \text{as $n\to\infty$}.
$$
Note that the natural extension map $\mathcal{T}_G$ is not defined on $\bigcup_{n=1}^{\infty} D_{\alpha,n}= [\alpha ,G)\times(1,G]$. We extend the
definition of on the set in the obvious way:
$$
\mathcal{T}_G(t,v) = \left( \frac{1}{t} -1 , \frac{1}{1+v}\right) ,\quad \text{for $(t,v)\in [\alpha ,G)\times(1,G]$}.
$$
At every ``time'' $n$ in our construction we add $M_{\alpha}(D_{\alpha,n})$ and $M_{\alpha}^2(D_{\alpha,n})$ to $\Omega_{\alpha,n}$, and remove from $\Omega_{\alpha,n}$ the rectangles $D_{\alpha,n}$ and $\mathcal{T}_G(D_{\alpha,n})$. Note that this last set is for $n\geq 0$ equal to:
$$
\mathcal{T}_G(D_{\alpha,n}) = \Big( -g^2, \frac{1}{\alpha}-1\Big] \times \left[ \frac{1}{1+m_{n+1}}, \frac{1}{1+m_n}\right].
$$
Note that the sequence $(\frac{1}{1+m_n})_{n\geq 0}$ is a monotonically decreasing and bounded sequence from below by $g^2$, and that we have that
$$
\lim_{n\to\infty} \frac{1}{1+m_n} = \frac{1}{1+G} = g^2.
$$
Therefore the set
$$
\bigcup_{n=0}^{\infty} \mathcal{T}_G(D_{\alpha,n}) = \mathcal{T}_G([\alpha,G)\times [0,G]) = (-g^2, \tfrac{1}{\alpha}-1]\times [g^2,1];
$$
will be removed from $\Omega_G$ in the limit as $n\to\infty$.

So repeating for $n\geq 0$ the process of adding $M_{\alpha}(D_{\alpha,n})$ and $M_{\alpha}^2(D_{\alpha,n})$ and deleting $D_{\alpha,n}$ and $\mathcal{T}_G(D_{\alpha,n})$, we find in the limit as $n\to\infty$ that:
$$
\Omega_{\alpha} = [\alpha -2, \frac{1}{\alpha}-1)\times [0,g^2]\cup [ \frac{1}{\alpha}-1, T_\alpha(\alpha-2))\times [0,1] \cup [ T_\alpha(\alpha-2), \alpha )\times [0,G];
$$
see Figure~\ref{fig:Omegaalpha}. This is exactly the domain of the natural extension of the map $T_{\alpha}$ for $1< \alpha < G$, as found by Boca and the fourth author in~\cite{[BM]}.

\begin{figure}[ht]
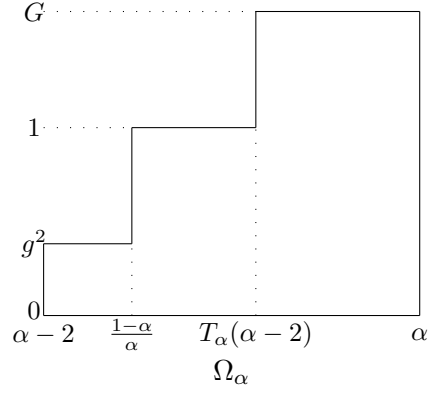

\beginpicture 

\setcoordinatesystem units <0.125cm,0.125cm> 
\setplotarea x from -45 to 30, y from -5 to 35

\putrule from -14 0 to 26 0
\putrule from 26 0 to 26 32.36068
\putrule from 8.571429 32.36068 to 26 32.36068
\putrule from 8.571429 20 to 8.571429 32.36068
\putrule from -4.615385 20 to 8.571429 20
\putrule from -4.615385 7.63932 to -4.615385 20
\putrule from -14 7.63932 to -4.615385 7.63932
\putrule from -14 0 to -14 7.63932

\put {$0$} at -15 0.5
\put {$g^2$} at -15 7.63932
\put {$1$} at -15 20
\put {$G$} at -15 32.36068

\put {$\alpha-2$} at -14 -2
\put {$\frac{1-\alpha}{\alpha}$} at -4.615385 -2
\put {$ T_\alpha(\alpha-2)$} at 8.571429 -2
\put {$\alpha$} at 26 -2

\put {\large $\Omega_{\alpha}$} at 6 -6

\setdots
\putrule from -14 20 to -4.615385 20
\putrule from -14 32.36068 to 8.571429 32.36068
\putrule from -4.615385 0 to -4.615385 7.63932
\putrule from 8.571429 0 to 8.571429 20

\endpicture
\caption{The domain $\Omega_\alpha$ for $\alpha\in (1,G)$.}
\label{fig:Omegaalpha}
\end{figure}

\section{The case $\alpha\in [\frac{ \sqrt{13}-1}{6},g)$}\label{subsec:smalleralpha}
In~\cite{[BM]}, Boca and the fourth author also found by simulation the natural extension for $\alpha = 0.9g$, which is a value of $\alpha$ outside their domain of possible $\alpha$. In this section we will use the approach from the previous section to exactly determine the domain of natural extension of ${\alpha}$ for $\alpha \in [\tfrac{ \sqrt{13}-1}{6},g)$. We will see that these natural extensions are quite more complicated as those for $\alpha\in [g,G]$, which are essentially the unions of 2 or 3 rectangles in $\R^2$. Again, as in the previous section, the natural extensions we find will be isomorphic the natural extension due to Schweiger and Rieger (so the case $\alpha = 1$).\smallskip\

For $\alpha = g$, Boca and the fourth author found in~\cite{[BM]}
$$
\Omega_g = \left( [g-2,g)\times [0,g^2]\right) \cup \left( [ \frac{-g^2}{1+g^2}, g) \times [1,G]\right)
$$
as planar domain for the natural extension $(\Omega_g, \bar{\mathcal{B}}_g, \bar{\mu}_g, \mathcal{T}_g)$; see Figure~\ref{fig:Omegasmalg}.

\begin{figure}[ht]
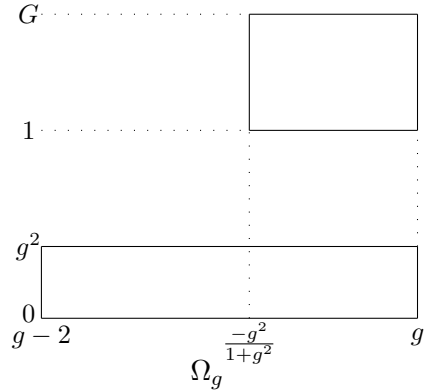

\beginpicture

\setcoordinatesystem units <0.125cm,0.125cm> 
\setplotarea x from -60 to 20, y from -5 to 25

\putrule from -27.63932 0 to 12.36068 0
\putrule from 12.36068 0 to 12.36068 7.63932
\putrule from -27.63932 7.63932 to 12.36068 7.63932
\putrule from -27.63932 0 to -27.63932 7.63932

\putrule from -5.527864 20 to 12.36068 20
\putrule from 12.36068 20 to 12.36068 32.36068
\putrule from -5.527864 32.36068 to 12.36068 32.36068
\putrule from -5.527864 20 to -5.527864 32.36068

\put {$0$} at -29 0.5
\put {$g^2$} at -29 7.63932
\put {$1$} at -29 20
\put {$G$} at -29 32.36068

\put {$g-2$} at -27.63932 -2
\put {$\frac{-g^2}{1+g^2}$} at -5.527864 -2.6
\put {$g$} at 12.36068 -2

\put {\large $\Omega_g$} at -10 -6

\setdots
\putrule from -27.63932 20 to -5.527864 20
\putrule from -27.63932 32.36068 to -5.527864 32.36068

\putrule from -5.527864 0 to -5.527864 20
\putrule from 12.36068 7.63932 to 12.36068 20

\endpicture
\caption{The domain $\Omega_{g}$.}
\label{fig:Omegasmalg}
\end{figure}

Now let $\alpha \in [\frac{1}{2}, g)$. Since $[\alpha, g) \not\subset I_{\alpha}$, we must consider those $(t_n,v_n)\in\Omega_g$ for which $t_n\geq \alpha$. So if $x\in I_g$ has $g$-expansion $x=[0;\varepsilon_1/a_1,\dots]$ and $n\geq 0$ is the first index for which $t_n\geq \alpha$, then $\varepsilon_{n+1} = +1$, $a_{n+1} = 3$, $\varepsilon_{n+2}=-1=\varepsilon_{n+3}$, $a_{n+2}=1$, and $a_{n+3}\geq 5$:
\begin{equation}\label{eq:Tg1expansionofx}
x = \frac{\varepsilon_1}{a_1+\ddots +\displaystyle{\frac{\varepsilon_n}{a_n+ \displaystyle{\frac{1}{ 3+ \displaystyle{\frac{-1}{1+\displaystyle{\frac{-1}{a_{n+3}+\ddots}}}}}}}}} ,
\end{equation}
Removing the $-1/1$ directly after $a_{n+1} (=3)$ in~(\ref{eq:Tg1expansionofx}) yields as an ``intermediate'' expansion of $x$:
\begin{equation}\label{eq:Tg1expansionofxpntermediate}
x = \frac{\varepsilon_1}{a_1+\ddots +\displaystyle{\frac{\varepsilon_n}{a_n + \displaystyle{\frac{1}{2+ \displaystyle{\frac{-1}{a_{n+3}-1+\ddots}}}}}}} .
\end{equation}
Note that in~(\ref{eq:Tg1expansionofxpntermediate}) we have that $a_{n+1}-1 = 2$, and that $a_{n+3}-1$ is even, so~(\ref{eq:Tg1expansionofxpntermediate}) is \emph{not} a continued fraction expansion of $x$ with only odd partial quotients. Inserting $1/1$ ``between'' $a_{n+1}-1(=2)$ and $\frac{-1}{a_{n+3}-1+\dots}$ (in fact, this is undoing an earlier singularization of $1/1$) yields that
\begin{equation}\label{eq:Tg1expansionofxalpha}
x = \frac{\varepsilon_1}{a_1+\ddots +\displaystyle{\frac{\varepsilon_n}{a_n + \displaystyle{\frac{1}{1 +\displaystyle{\frac{1}{1+\displaystyle{\frac{1}{a_{n+3}-2+\ddots}}}}}}}}} .
\end{equation}
Now~(\ref{eq:Tg1expansionofxalpha}) is a continued fraction expansion of $x$ with only odd partial quotients. However, it is \textbf{not} the $\alpha$-expansion of $x$, since for these values of $\alpha$ we cannot have that a partial quotient is equal to 1 while its corresponding $\varepsilon_i=+1$. Singularizing now the first partial quotient 1 after $a_n$ yields:
\begin{equation}\label{eq:Tg1expansionofxalphaalmostthere}
x = \frac{\varepsilon_1}{a_1+\ddots +\displaystyle{\frac{\varepsilon_n}{a_n + 1 + \displaystyle{\frac{-1}{2+\displaystyle{\frac{1}{a_{n+3}-2+\ddots}}}}}}} .
\end{equation}
We find that (\ref{eq:Tg1expansionofxalphaalmostthere}) is a continued fraction expansion of $x$ of which some of the partial quotients are even. Inserting $-1/1$ after $a_n+1$ now yields:
\begin{equation}\label{eq:Tg1expansionofxalphafinallythere}
x = \frac{\varepsilon_1}{a_1+\ddots +\displaystyle{\frac{\varepsilon_n}{a_n + 2 + \displaystyle{\frac{-1}{1+ \displaystyle{\frac{-1}{3+\displaystyle{\frac{1}{a_{n+3}-2+\ddots}}}}}}}}} .
\end{equation}
If we denote the expansion~(\ref{eq:Tg1expansionofxalphafinallythere}) of $x$ as $[a_0^*; \varepsilon_1^*/a_1^*,\varepsilon_2^*/a_2^*,\dots ]$, we find that
\begin{equation}\label{eq:t_n^*}
t_n^* = [0;-1/1,-1/3,1/a_{n+3}-2,\dots]
\end{equation}
and
\begin{equation}\label{eq:v_n^*}
v_n^* = [0;1/a_n+2,\varepsilon_n/a_{n-1},\dots,\varepsilon_2/a_1].
\end{equation}
Recall that for the $g$-expansion of $x$ we have that
$$
t_n = [0; 1/3, -1/1, -1/a_{n+3}, \dots ]
$$
and
$$
v_n = [0; 1/a_n, \varepsilon_n/a_{n-1},\dots,\varepsilon_2/a_1].
$$
But then we have, that:
\begin{eqnarray*}
t_n &=& \frac{1}{3+\displaystyle{\frac{-1}{1+\displaystyle{\frac{-1}{a_{n+3} + \ddots}}}}} \,\, =\,\, \frac{1}{2+\displaystyle{\frac{-1}{a_{n+3}-1+\ddots}}} \,\,
=\,\,  \frac{1}{1+\displaystyle{\frac{1}{1+\displaystyle{\frac{1}{a_{n+3}-2+\ddots}}}}} \\
&=& 1+\frac{-1}{2+\displaystyle{\frac{1}{a_{n+3}-2+\ddots}}} \,\, =\,\, 1+1+\frac{-1}{1+\displaystyle{\frac{-1}{3+\displaystyle{\frac{1}{a_{n+3}-2+\ddots}}}}} \,\, = \,\, 2+t_n^*,
\end{eqnarray*}
from which we see that $t_n^* = t_n-2$ (as was to be expected).

As in the previous sections we also have, since $v_{n-1}=q_{n-2}/q_{n-1}$, that
$$
v_n^* = \frac{v_n}{1+2v_n},
$$
and we see that
$$
(t_n^*,v_n^*) = \left( t_n-2,\frac{v_n}{1+2v_n}\right) ,
$$
and in view of this we define sets $D_{\alpha,\ell} = [\alpha ,g)\times [0,g^2]$, $D_{\alpha ,u} = [\alpha , g)\times [1,G]$, and a map $M_{\alpha}: \Omega_g\to\R^2$ by
$$
M_{\alpha} (t,v) = \begin{cases}
(t,v), & \text{if $(t,v)\not\in D_{\alpha ,\ell}\cup D_{\alpha ,u}$}; \\
\left( t-2, \frac{v}{2v+1}\right), & \text{if $(t,v)\in D_{\alpha ,\ell}\cup D_{\alpha ,u}$};
\end{cases}
$$
see Figure~\ref{fig:Omegagomegaalpha}. We now ``delete'' $D_{\alpha ,\ell}$ and $D_{\alpha ,u}$ from $\Omega_g$, and add the new regions $M_{\alpha}(D_{\alpha ,\ell})$ and $M_{\alpha}(D_{\alpha ,u})$ to $\Omega_g$. These regions are:
$$
M_{\alpha}(D_{\alpha ,\ell}) = [\alpha -2, g-2]\times \left[ 0, \frac{1}{3+G}\right] ,\, \text{and}\quad
M_{\alpha}(D_{\alpha ,u}) = [\alpha -2, g-2]\times \left[ \frac{1}{3},g^2\right].
$$

Note that by ``removing'' $D_{\alpha ,\ell}$ and $D_{\alpha ,u}$ we create ``holes'' in the lower rectangle in $\Omega_g$:
$$
\mathcal{T}_g(D_{\alpha ,\ell}) = \left[ g-2,\frac{1-3\alpha}{\alpha}\right] \times \left[ \frac{1}{4-g},\frac{1}{3}\right]
$$
and
$$
\mathcal{T}_g(D_{\alpha ,u}) = \left[ g-2,\frac{1-3\alpha}{\alpha}\right] \times \left[ \frac{1}{3+G},\frac{1}{4}\right] .
$$
In their turn, these hole create holes in the upper rectangle of $\Omega_g$:
$$
\mathcal{T}_g^2(D_{\alpha ,\ell}) =  \left[ \frac{-g^2}{1+g^2} , \frac{1-2\alpha}{3\alpha -1} \right] \times \left[ \frac{4-g}{3-g},\tfrac{3}{2}\right] ,
$$
and
$$
\mathcal{T}_g^2(D_{\alpha ,u}) = \left[ \frac{-g^2}{1+g^2} , \frac{1-2\alpha}{3\alpha -1} \right] \times
\left[ \frac{3+G}{2+G}, \tfrac{4}{3} \right] .
$$
Note that for $\alpha \in [\frac{\sqrt{13}-1}{6},g)$ we have that
$$
\frac{-g^2}{1+g^2} < \frac{1-2\alpha}{3\alpha -1} \leq \alpha ,
$$
and that $\frac{1-2\alpha}{3\alpha -1} = \alpha$ yields that $\alpha = \frac{ \sqrt{13}-1}{6} = 0.434258545\cdots$; for this value of $\alpha$ we have that the ``upper rectangle'' has two ``holes'' running all the way through the rectangle.

\begin{figure}[ht]
\beginpicture 

\setcoordinatesystem units <0.125cm,0.125cm> 
\setplotarea x from -30 to 90, y from -10 to 30

\putrule from -27.63932 0 to 12.36068 0
\putrule from 12.36068 0 to 12.36068 7.63932
\putrule from -27.63932 7.63932 to 12.36068 7.63932
\putrule from -27.63932 0 to -27.63932 7.63932

\putrule from 10.5 0 to 10.5 7.63932

\putrule from -27.63932 5.91372 to -21.90476 5.91372
\putrule from -21.90476 5.91372 to -21.90476 6.66667
\putrule from -27.63932 6.66667 to -21.90476 6.66667

\putrule from -27.63932 4.330847 to -21.90476 4.330847
\putrule from -21.90476 4.330847 to -21.90476 5
\putrule from -27.63932 5 to -21.90476 5

\putrule from -5.527864 20 to 12.36068 20
\putrule from 12.36068 20 to 12.36068 32.36068
\putrule from -5.527864 32.36068 to 12.36068 32.36068
\putrule from -5.527864 20 to -5.527864 32.36068

\putrule from 10.5 20 to 10.5 32.36068

\put {$0$} at -29 0.5
\put {$g^2$} at -29 7.63932
\put {$1$} at -29 20
\put {$G$} at -29 32.36068

\put {$g-2$} at -28.63932 -2
\put {$\frac{1-3\alpha}{\alpha}$} at -20.90476 -2.5
\put {$\frac{-g^2}{1+g^2}$} at -5.527864 -2.5
\put {$\alpha$} at 10.45 -1.5
\put {$g$} at 12.55 -1.5

\putrule from 26.5 0 to 66.5 0
\putrule from 66.5 0 to 66.5 7.63932
\putrule from 26.5 7.63932 to 66.5 7.63932
\putrule from 26.5 6.66667 to 26.5 7.63932
\putrule from 26.5 6.66667 to 34.09524 6.66667
\putrule from 34.09524 6.66667 to 34.09524 5.91372
\putrule from 28.36068 5.91372 to 34.09524 5.91372
\putrule from 28.36068 5 to 28.36068 5.91372
\putrule from 28.36068 5 to 34.09524 5
\putrule from 34.09524 4.330847 to 34.09524 5
\putrule from 26.5  4.330847 to 34.09524 4.330847
\putrule from 26.5 0 to 26.5 4.330847

\putrule from 48.47214 20 to 66.5 20
\putrule from 66.5 20 to 66.5 32.36068
\putrule from 48.47214 32.36068 to 66.5 32.36068
\putrule from 48.47214 30 to 48.47214 32.36068
\putrule from 48.47214 30 to 54.261 30
\putrule from 54.261 30 to 54.261 27.639
\putrule from 50.47214 27.639 to 54.261 27.639
\putrule from 50.47214 26.666 to 50.47214 27.639
\putrule from 50.47214 26.666 to 54.261 26.666
\putrule from 54.261 25.528 to 54.261 26.666

\putrule from 48.47214 25.528 to 54.261 25.528
\putrule from 48.47214 20 to 48.47214 25.528

\putrule from 58.1052 13.333 to 68.36068 13.333
\putrule from 68.36068 13.333 to 68.36068 14.472
\putrule from 58.1052 14.472 to 68.36068 14.472
\putrule from 58.1052 13.333 to 58.1052 14.472

\putrule from 58.1052 10 to 68.36068 10
\putrule from 68.36068 10 to 68.36068 11.6035
\putrule from 58.1052 11.6035 to 68.36068 11.6035
\putrule from 58.1052 10 to 58.1052 11.6035

\put {$1$} at 25 20
\put {$G$} at 25 32.36068

\put {\tiny $\alpha-2$} at 26.5 -1.5
\put {\tiny $g-2$} at 28.36068 -2.5
\put {$\frac{1-3\alpha}{\alpha}$} at 34.09524 -2.5
\put {\tiny $T_\alpha(\alpha-2)$} at 46 34
\put {$\frac{-g^2}{1+g^2}$} at 50.47214 -2.5
\put {$\frac{1-2\alpha}{3\alpha-1}$} at 54.261 34.3
\put {$\frac{2\alpha-1}{1-\alpha}$} at 58.1052 -2.5
\put {$\alpha$} at 66 -1.5
\put {$g$} at 68.5 -1.5

\put {\large $\Omega_g$} at -10 -6
\put {\large $\Omega_{\alpha,1}$ } at 45 -6
\setdots
\putrule from -27.63932 20 to -5.527864 20
\putrule from -27.63932 32.36068 to -5.527864 32.36068

\putrule from -5.527864 0 to -5.527864 20
\putrule from 12.36068 7.63932 to 12.36068 20

\putrule from 10.5 7.63932 to 10.5 20
\putrule from -21.90476 0 to -21.90476 7.63932

\putrule from 28.36068 0 to 28.36068 7.63932
\putrule from 34.09524 0 to 34.09524 7.63932
\putrule from 48.47214 20 to 48.47214 32.5
\putrule from 50.47214 0 to 50.47214 30
\putrule from 54.261 20 to 54.261 32.5
\putrule from 58.1052 0 to 58.1052 20
\putrule from 66.5 7.63932 to 66.5 20
\putrule from 68.36068 0 to 68.36068 14.472

\putrule from 26.5 32.36068 to 50.47214 32.36068
\putrule from 26.5 20 to 50.47214 20

\endpicture
\caption{The domain $\Omega_{g}$ and the region $\Omega_{\alpha,1}$.}
\label{fig:Omegagomegaalpha}
\end{figure}

From~(\ref{eq:t_n^*}) it follows that
\begin{equation}\label{eq:t_n+1^*}
\mathcal{T}_{\alpha}(t_n^*) = \frac{-1}{t_n^*}-1 = [0;-1/3,1/a_{n+2}-2,\dots] = t_{n+1}^*,
\end{equation}
and from~(\ref{eq:v_n^*}) we see that
\begin{equation}\label{eq:v_n+1^*}
v_{n+1}^* = [0;1/1, -1/a_n+2,\varepsilon_n/a_{n-1},\dots,\varepsilon_2/a_1] = \frac{1}{1-v_n^*} ;
\end{equation}
so it follows that
$$
(t_{n+1}^*,v_{n+1}^*) = \mathcal{T}_{\alpha}(t_n^*,v_n^*).
$$
We have that two rectangles with ``new'' areas are attached to the ``upper rectangle'' in $\Omega_g$:
$$
N_{\alpha,\ell} := \mathcal{T}_{\alpha}(M_{\alpha}(D_{\alpha,\ell})) = \left[ T_{\alpha}(\alpha -2),\frac{-g^2}{1+g^2}\right] \times \left[ 1,\frac{3+G}{2+G}\right]
$$
and
$$
N_{\alpha,u} := \mathcal{T}_{\alpha}(M_{\alpha}(D_{\alpha,u})) = \left[ T_{\alpha}(\alpha -2),\frac{-g^2}{1+g^2}\right] \times \left[ \frac{3}{2},G\right] .
$$
In view of~(\ref{eq:t_n+1^*}) and~(\ref{eq:v_n+1^*}) we now apply the auxiliary map
$$
\mathcal{A}(\xi,\eta ) = \left( -\frac{1}{\xi}-3, \frac{1}{3-\eta}\right)
$$
to $N_{\alpha,\ell}$ and $N_{\alpha,u}$ (note that on these two regions this map is actually $\mathcal{T}_g$), yielding two disjoint ``islands'' (and here we used that $T_{\alpha}(\alpha -2)= \frac{\alpha -1}{2-\alpha}$):
$$
\mathcal{A}( N_{\alpha,\ell}) = \left[ \frac{2\alpha -1}{1-\alpha}, g \right] \times \left[ \frac{1}{2}, \frac{2+G}{3+2G}\right]
$$
and
$$
\mathcal{A}( N_{\alpha,u}) = \left[ \frac{2\alpha -1}{1-\alpha}, g \right] \times \left[ \frac{2}{3},\frac{1}{3-G} \right] ;
$$
see Figure~\ref{fig:Omegagomegaalpha}. Note in particular how well the ``new part'' and the ``holes'' align.

As in Section~\ref{subsec:grotesquesystem}, the domain $\Omega_{\alpha,1}$, given by
\begin{eqnarray*}
\Omega_{\alpha,1} &=& \Omega_g\setminus (\mathcal{T}_g(D_{\alpha, \ell}\cup D_{\alpha, u}) \cup \mathcal{T}_g^2(D_{\alpha, \ell}\cup D_{\alpha, u})) \\
& &\,\, \cup\, M_{\alpha}(D_{\alpha, \ell}\cup D_{\alpha, u}) \cup \mathcal{T}_{\alpha}(M_{\alpha}(D_{\alpha, \ell}\cup D_{\alpha, u})) \\
& &\,\, \cup\, \mathcal{A}(\mathcal{T}_{\alpha}(M_{\alpha}(D_{\alpha, \ell}\cup D_{\alpha, u})))
\end{eqnarray*}
is \textbf{not} a version of the natural extension for $\alpha$; the reason is, that since $T_{\alpha}(\alpha -2)\leq \alpha$ the two rectangles $\mathcal{A}( N_{\alpha,\ell})$ and $\mathcal{A}( N_{\alpha,u})$ have values on their first coordinates which are larger than $\alpha$; see also Figure~\ref{fig:Omegagomegaalpha}. In view of this we define new regions $D_{\alpha,\ell ,1}$ and $D_{\alpha,u,1}$ which must both be moved and deleted in a way similar to what we just did:
$$
D_{\alpha,\ell ,1} = [\alpha ,g)\times \left[ \frac{1}{2}, \frac{2+G}{3+2G}\right]
,\, \text{and}\,\,
D_{\alpha,u,1} = [\alpha ,g)\times  \left[ \frac{2}{3},\frac{1}{3-G} \right] .
$$
So as in Section~\ref{subsec:grotesquesystem} we need to repeat (infinitely often) the above procedure, but now to $D_{\alpha,\ell ,1}$ and $D_{\alpha,u,1}$. We will describe one more ``round,'' and then give the general situation, which can be proved by induction. Extending the domains of the maps we defined in this Section ($M_{\alpha}$ and $\mathcal{A}$) in the obvious way, we find that two new added regions are:
$$
M_{\alpha}(D_{\alpha,\ell ,1}) = [\alpha -2,g-2) \times \left[ \frac{1}{4},\frac{G+2}{4G+7}\right]
$$
and
$$
M_{\alpha}(D_{\alpha,u ,1}) = [\alpha -2,g-2) \times \left[ \frac{2}{7},\frac{1}{4-g}\right]
$$
(note that these new rectangles align well with the existing ``holes''). By ``removing'' $D_{\alpha ,\ell ,1}$ and $D_{\alpha ,u,1}$ we create new ``holes'' in the lower rectangle in $\Omega_g$:
$$
\mathcal{T}_g(D_{\alpha ,\ell ,1}) = \left[ g-2,\frac{1-3\alpha}{\alpha}\right] \times \left[\frac{2G+2}{7G+11},\frac{2}{7} \right]
$$
and
$$
\mathcal{T}_g(D_{\alpha ,u,1}) = \left[ g-2,\frac{1-3\alpha}{\alpha}\right] \times \left[ \frac{2-g}{7-3g},\frac{3}{11}\right] .
$$
Since $\frac{2-g}{7-3g}=0.26855\cdots = \frac{G+2}{4G+7}$ we see again an alignment between new rectangles and new holes. In their turn, these new holes create holes in the upper rectangle of $\Omega_g$:
$$
\mathcal{T}_g^2(D_{\alpha ,\ell ,1}) =  \left[ \frac{-g^2}{1+g^2} , \frac{1-2\alpha}{3\alpha -1} \right] \times \left[ \frac{7G+11}{5G+8}, \frac{7}{5}\right] ,
$$
and
$$
\mathcal{T}_g^2(D_{\alpha ,u,1}) = \left[ \frac{-g^2}{1+g^2} , \frac{1-2\alpha}{3\alpha -1} \right] \times
\left[ \frac{4G+7}{3G+5}, \frac{11}{8}\right] .
$$
Again we have that two rectangles with ``new'' areas are attached to the `upper rectangle' in $\Omega_g$:
$$
N_{\alpha,\ell ,1} := \mathcal{T}_{\alpha}(M_{\alpha}(D_{\alpha,\ell})) = \left[ T_{\alpha}(\alpha -2),\frac{-g^2}{1+g^2}\right] \times \left[ \frac{2G+1}{2G}, \frac{4}{3}\right]
$$
and
$$
N_{\alpha,u,1} := \mathcal{T}_{\alpha}(M_{\alpha}(D_{\alpha,u})) = \left[ T_{\alpha}(\alpha -2),\frac{-g^2}{1+g^2}\right] \times \left[ \frac{7}{5}, \frac{4-g}{3-g}\right] .
$$
Applying the auxiliary map $\mathcal{A}$ to these new rectangles $N_{\alpha,\ell ,1}$ and $N_{\alpha,u,1}$ yields two disjoint ``islands'' located in between the previous two ``islands'':
$$
\mathcal{A}( N_{\alpha,\ell ,1}) = \left[ \frac{2\alpha -1}{1-\alpha}, g \right] \times \left[ \frac{2}{4-g},\frac{3}{5} \right]
$$
and
$$
\mathcal{A}( N_{\alpha,u,1}) = \left[ \frac{2\alpha -1}{1-\alpha}, g \right] \times \left[ \frac{3-g}{5-2g},\frac{5}{8} \right] .
$$
As for $\Omega_{\alpha ,1}$, the domain $\Omega_{\alpha,2}$, given by
\begin{eqnarray*}
\Omega_{\alpha,2} &=& \Omega_{\alpha ,1}\setminus (\mathcal{T}_g(D_{\alpha, \ell ,1}\cup D_{\alpha, u,1}) \cup \mathcal{T}_g^2(D_{\alpha, \ell ,1}\cup D_{\alpha, u,1})) \\
& &\,\, \cup\, M_{\alpha}(D_{\alpha, \ell ,1}\cup D_{\alpha, u,1}) \cup \mathcal{T}_{\alpha}(M_{\alpha}(D_{\alpha, \ell ,1}\cup D_{\alpha, u,1})) \\
& &\,\, \cup\, \mathcal{A}(\mathcal{T}_{\alpha}(M_{\alpha}(D_{\alpha, \ell ,1}\cup D_{\alpha, u,1})))
\end{eqnarray*}
is \textbf{not} a version of the natural extension for $\alpha$; again two rectangles are ``sticking out'':
$$
D_{\alpha ,\ell ,2} = \left[ \alpha , g \right] \times \left[ \frac{2}{4-g},\frac{3}{5} \right]
\,\, \text{and}\,\,
D_{\alpha ,u,2} = \left[ \alpha , g \right] \times  \left[ \frac{3-g}{5-2g},\frac{5}{8} \right] .
$$
By repeating the above procedure, we define for $n\geq 1$ nested disjoint rectangles $D_{\alpha , \ell, n}$ and $D_{\alpha ,u, n}$.\medskip\

Note, that for $(u,v)\in D_{\alpha, \ell}\cup D_{\alpha , u}\cup D_{\alpha ,\ell , 1}\cup D_{\alpha, u,1}$ we have that
\begin{equation}\label{eq:regularexpansionV}
v\mapsto \frac{v}{2v+1}\mapsto \frac{1}{1-\frac{v}{2v+1}} = \frac{2v+1}{v+1} \mapsto \frac{1}{3-\frac{2v+1}{v+1}} = \frac{v+1}{v+2} =
\frac{1}{1+ {\displaystyle \frac{1}{1 + \displaystyle \frac{1}{1+v}}}},
\end{equation}
and we see that the function $V:[0,G]\to\R$, defined by
$$
V(x) = \frac{1+x}{2+x},
$$
is monotonically increasing, with fixed point $g$. Setting $D_{\alpha , \ell, 0}:=D_{\alpha , \ell}$ and $D_{\alpha ,u,0}:=D_{\alpha , u}$, we see for $n\geq 0$ that:
$$
D_{\alpha , \ell, n} = \Big[\alpha,g\Big) \times V^n\left( [0,g^2] \right),\,\,
\text{and}\,\,\,\,
D_{\alpha , u, n} = \Big[\alpha,g\Big) \times V^n\left( [1,G] \right) .
$$
Setting for $n\geq 0$:
$$
V^n\left( [0,g^2] \right) = [\ell_n,u_n],\quad \text{and}\quad  V^n\left( [1,G] \right) = [L_n,U_n],
$$
so $\ell_0=0$, $u_0=g^2$, $L_0=1$ and $U_0=G$, we see that it follows from~(\ref{eq:regularexpansionV}) that the regular continued fraction expansions of $\ell_n$, $u_n$, $L_n$ and $U_n$ can be explicitly given; for $n\geq 1$:
$$
\ell_n= [0; 1^{2n}]\,\,\, (= [0; \underbrace{1,1,\dots,1}_{2n-\text{times}}]),\quad u_n = [0;1^{2n}, 2, \bar{1}\, ],\quad L_n=[0;1^{2n+1}],
$$
(the bar indicates the periodic part; here we used that $g^2=[0;2,\overline{1}\, ]$), and
$$
U_n = [0; 1^{2n-1},2, \overline{1}\, ]
$$
(here we used that $G=[1;\overline{1}\, ] = 1+[0;\overline{1}\, ]$). Elementary properties of regular continued fraction expansions now yield that for $n\geq 0$:
$$
\ell_n < u_n < \ell_{n+1},\quad L_{n+1} < U_{n+1} < L_{n},
$$
and
$$
\ell_n\uparrow g,\quad \text{resp.} \quad L_n\downarrow g, \quad  \text{as $n\to\infty$}.
$$
Due to this, and due to the fact that the maps involved ($M_{\alpha}$, $\mathcal{T}_g$, $\mathcal{T}_{\alpha}$ and $\mathcal{A}$) are all bijective a.s.\ we see that $(M_{\alpha}(D_{\alpha ,\ell , n}))_n$ is a disjoint sequence of rectangles, converging from below to the line segment $[\alpha -2,g-2)\times \{ \frac{1}{G+2} \}$, and that $(M_{\alpha}(D_{\alpha ,u, n}))_n$ is a disjoint sequence of rectangles, also converging (but now from above) to the same line segment $[\alpha -2,g-2)\times \{ \frac{1}{G+2} \}$. Due to this we have that $(\mathcal{T}_{\alpha}(M_{\alpha}(D_{\alpha ,\ell , n})))_n$ is a disjoint sequence of rectangles, converging from below to the line segment $[T_{\alpha}(\alpha -2),T_g(g-2))\times \{ 1+g^2 \}$, and that $(T_{\alpha}(M_{\alpha}(D_{\alpha ,u, n})))_n$ is a disjoint sequence of rectangles, also converging (but now from above) to the same line segment $[T_{\alpha}(\alpha -2),T_g(g-2))\times \{ 1+g^2 \}$.\medskip\

For $n\geq 1$ we have that:
$$
M_{\alpha}(D_{\alpha ,\ell ,n}) = [ \alpha -2,g-2 ) \times \left[ \frac{\ell_n}{2\ell_n+1},\frac{u_n}{2u_n+1}\right] ,
$$
where by definition of $M_{\alpha}$ we have that
$$
\frac{\ell_n}{2\ell_n+1} = \frac{1}{2+ {\displaystyle \frac{1}{\ell_n}}} = [0;3,1^{2n-1}].
$$
In the same way we see that
$$
\frac{u_n}{2u_n+1} = [0;3,1^{2n-1},2,\overline{1}\, ], \quad \frac{L_n}{2L_n+1} = [0;3,1^{2n}], \quad  \frac{U_n}{2U_n+1} = [0;3,1^{2n-2},2,\overline{1}\, ].
$$

Using~(\ref{eq:regularexpansionV}) we can obtain in a similar way the regular continued fraction expansions of all endpoints of the second coordinate interval of all of the other added or deleted rectangles. For example, for $n\geq 1$ we have that
$$
\mathcal{T}_g(D_{\alpha,\ell ,n}) = \Big( g-2,\frac{1-3\alpha}{\alpha}\Big] \times \left[ \frac{1}{3+u_n},\frac{1}{3+\ell_n}\right]
$$
and
$$
\mathcal{T}_g(D_{\alpha,u,n}) = \Big( g-2,\frac{1-3\alpha}{\alpha}\Big] \times \left[ \frac{1}{3+U_n},\frac{1}{3+L_n}\right],
$$
where
$$
\frac{1}{3+u_n}= [0;3,1^{2n},2,\overline{1}],\quad \frac{1}{3+\ell_n} = [0;3,1^{2n}],
$$
and
$$
\frac{1}{3+U_n}= [0;3,1^{2n-1},2,\overline{1}],\quad \frac{1}{3+L_n} = [0;3,1^{2n+1}].
$$
Recall we already saw that:
$$
\mathcal{T}_g(D_{\alpha,\ell}) = \Big( g-2,\frac{1-3\alpha}{\alpha}\Big] \times \left[ \frac{1}{4-g},\frac{1}{3}\right]
$$
and that
$$
\mathcal{T}_g(D_{\alpha,u}) = \Big( g-2,\frac{1-3\alpha}{\alpha}\Big] \times \left[ \frac{1}{3+G},\frac{1}{4}\right].
$$
Note that the added intervals and the deleted intervals ``align'' (see also Figure~\ref{fig:Omegagomegaalphazoom}).

\begin{figure}[ht]
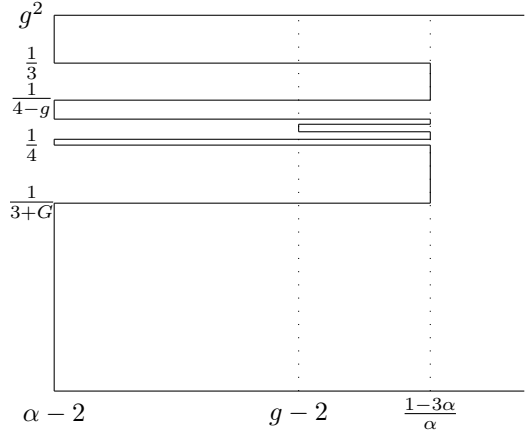

\beginpicture

\setcoordinatesystem units <0.25cm,0.25cm> 
\setplotarea x from -10 to 30, y from -2 to 22

\putrule from 0 0 to 25 0
\putrule from 0 0 to 0 10
\putrule from 0 10 to 20 10
\putrule from 20 10 to 20 13.090
\putrule from 0 13.090 to 20 13.090
\putrule from 0 13.090 to 0 13.4
\putrule from 0 13.4 to 20 13.4

\putrule from 20 13.4 to 20 13.8
\putrule from 13 13.8 to 20 13.8
\putrule from 13 13.8 to 13 14.2
\putrule from 13 14.2 to 20 14.2
\putrule from 20 14.2 to 20 14.472

\putrule from 0 14.472 to 20 14.472
\putrule from 0 15.482 to 0 14.472
\putrule from 0 15.482 to 20 15.482
\putrule from 20 15.482 to 20 17.453
\putrule from 0 17.453 to 20 17.453
\putrule from 0 17.453 to 0 20 
\putrule from 0 20 to 25 20

\put {$\alpha-2$} at 0 -1.2
\put {$g-2$} at 13 -1.2
\put {$\frac{1-3\alpha}{\alpha}$} at 20 -1.2

\put {$\frac{1}{3+G}$} at -1.2  10 
\put {$\frac{1}{4}$} at -1.2 13.090
\put {$\frac{1}{4-g}$} at -1.2 15.482 
\put {$\frac{1}{3}$} at -1.2 17.453
\put {$g^2$} at -1.2 20 

\setdots
\putrule from 13 0 to 13 20
\putrule from 20 0 to 20 20

\endpicture
\caption{The left part of the natural extension after ``one more round'' for $\alpha\in \frac{\sqrt{13}-1}{6} \leq \alpha < g$.}
\label{fig:Omegagomegaalphazoom}
\end{figure}

Setting $D_{\alpha , \ell,0}=D_{\alpha ,\ell}$ and $D_{\alpha , u,0}=D_{\alpha ,u}$, we have obtained for $\alpha \in [\tfrac{1}{2},g)$ the result stated in the following Theorem.

\begin{theorem}\label{thm:newNatExt}
Let $\frac{\sqrt{13}-1}{6} \leq \alpha < g$, then a version of the natural extension of the odd $\alpha$-continued fraction expansion map $T_{\alpha}$ from (\ref{eq:alphamap}) is given by
\begin{eqnarray*}
\Omega_{\alpha} &=& \Big( \bigcup_{n=0}^{\infty} M_{\alpha}(D_{\alpha ,\ell,n}\cup D_{\alpha ,u,n}) \Big) \cup  \Big( \bigcup_{n=0}^{\infty} \mathcal{T}_{\alpha}(M_{\alpha}(D_{\alpha ,\ell,n}\cup D_{\alpha ,u,n}))\Big) \\
& & \cup\,\,  \Big( [g-2,\alpha) \times [0,g^2] \Big) \setminus \Big( \bigcup_{n=0}^{\infty} \mathcal{T}_g(D_{\alpha ,\ell ,n}\cup D_{\alpha ,u,n}) \Big) \\
& &   \cup\,\,  \Big( [\frac{-1}{G+2},\alpha ) \times [1,G] \Big) \setminus \Big( \bigcup_{n=0}^{\infty} \mathcal{T}_g^2(D_{\alpha ,\ell ,n}\cup D_{\alpha ,u,n}) \Big) \\
& & \cup\,\,  \Big( \bigcup_{n=0}^{\infty} \mathcal{A}\Big( \mathcal{T}_{\alpha}(M_{\alpha}(D_{\alpha ,\ell,n}\cup D_{\alpha ,u,n}))\Big) \Big) .
\end{eqnarray*}
The dynamical system $(\Omega_{\alpha}, \mathcal{B}_{\alpha}, \mu_{\alpha}, \mathcal{T}_{\alpha})$, where $\mathcal{B}_{\alpha}$ is the collection of Borel subsets of $\Omega_{\alpha}$ and $\mu_{\alpha}$ is a probability measure on $(\Omega_{\alpha}, \mathcal{B}_{\alpha})$ with density
$$
d_{\alpha}(x,y) = \frac{1}{3\log G} \frac{1}{(1+xy)^2},\quad \text{for $(x,y)\in \Omega_{\alpha}$},
$$
and $d_{\alpha}(x,y)=0$ elsewhere, is ergodic and metrically isomorphic to the natural extension for any other $\alpha^*\in [\frac{\sqrt{13}-1}{6}, G]$. As a consequence, we have that the Kolmogorov-Sinai entropy equals $\pi^2/(9\log G)$ for these values of $\alpha$.
\end{theorem}

\begin{proof}[Proof of Theorem \ref{thm:newNatExt}] As stated before, we only need to prove the Theorem for $\alpha \in \Big[ \frac{\sqrt{13}-1}{6}, \tfrac{1}{2}\Big)$. We repeat the procedure we have used now several times, starting with $\alpha = \tfrac{1}{2}$, and choosing $\frac{\sqrt{13}-1}{6} \leq \alpha < \tfrac{1}{2}$. Since $T_{\frac{1}{2}}^2(\tfrac{1}{2}) = 0$, the situation is slightly different from the case where $\tfrac{1}{2}\leq \alpha < g$: if $x\in I_{\tfrac{1}{2}}$ has $\tfrac{1}{2}$-expansion $x = [0;\varepsilon_1/a_1,\dots ]$ and $n\geq 0$ is the first index for which $t_n\geq \alpha$, then $\varepsilon_{n+1} = +1$, $a_{n+1}=3$, $\varepsilon_{n+2}=-1$, $a_{n+2}=1$ and (and this is different from the case $\tfrac{1}{2}\leq \alpha < g$): $\varepsilon_{n+3} = +1$, $a_{n+3}\geq 3$. What is also different from the case $\tfrac{1}{2}\leq \alpha < g$, is that we only need to do the procedure once; now
$$
D_{\alpha} = [\alpha ,\tfrac{1}{2}) \times \bigcup_{n=0}^{\infty} \left( [\ell_n,u_n]\cup [L_n,U_n]\right).
$$
Details are left to the reader.
\end{proof}

\begin{remark}
{\rm Note that we must have in Theorem~\ref{thm:newNatExt} that $\frac{\sqrt{13}-1}{6} \leq \alpha < g$; in case $\alpha = \frac{\sqrt{13}-1}{6}$ one can see that $T_g^2(\alpha )=\alpha$, and that the ``hole'' in the top rectangle of $\Omega_g$ will be ``all the way through.'' Clearly a new situation arises for values of $\alpha$ smaller than $\frac{\sqrt{13}-1}{6}$, which can be addressed by our approach, but things are becoming even more tedious for such values.}
\end{remark}

\section{Entropy}\label{sec:ent}
Entropy as a function of $\alpha$ is widely studied for several different families of continued fractions. Nakada's $\alpha$-continued fractions~\cite{[CIT],[LM],[N]}, Ito and Tanaka's $\alpha$-continued fractions~\cite{[CLS]} and Katok and Ugarcovici's $\alpha$-continued fractions~\cite{[KUa],[KUb]} are the best known examples. Various tools work in a similar way for these families (see~\cite{[L]}, Section 3 for a general set up). In this section we will compare the entropy function of Odd $\alpha$-continued fractions with the entropy function of Nakada's $\alpha$-continued fractions (see Figure~\ref{fig:ent} for both functions). Nakada's $\alpha$-continued fractions are the most studied. They were introduced in 1981 by Nakada~\cite{[N]} and brought back under attention by Luzzi and Marmi in 2008 (see~\cite{[LM]}). For the odd $\alpha$-continued fraction the entropy is explicitly known on the interval $[ \frac{\sqrt{13}-1}{6} ,G]$ and equals $\frac{\pi^2}{9\log G}$. This value is found in~\cite{[BM]} for the interval $[g,G]$ and follows from the previous section for $\alpha\in( \frac{ \sqrt{13}-1}{6} ,g)$. For Nakada's $\alpha$-continued fraction the entropy is known on $[g^2,1]$ and is given by
$h(\alpha)=\frac{\pi^2}{6\log(1+\alpha)}$ for $\alpha\in(g,1]$ and $h(\alpha)=\frac{\pi^2}{6\log(G)}$ for $\alpha\in[g^2,g]$; see~\cite{[KSS],[N]}. Note that for both families the entropy is not explicitly known on a large part of the parameter space, but obtained/estimated by simulation.

\begin{figure}[h!]
  \centering
{\includegraphics[width=0.45\textwidth]
{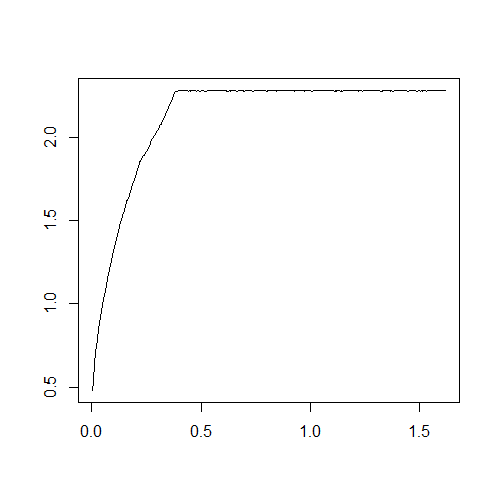}}
  \hfill
{\includegraphics[width=0.45\textwidth]
{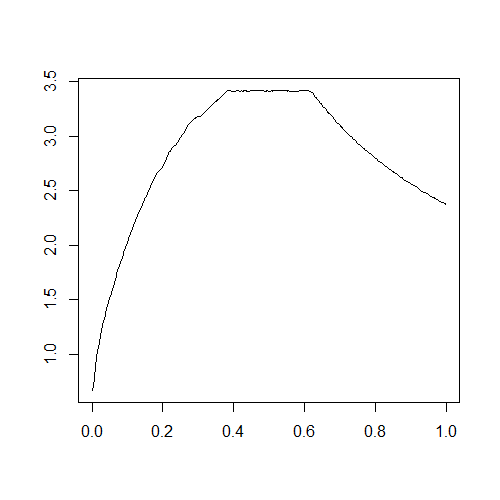}}
  \caption{The entropy plotted as a function of $\alpha$. On the left for the Odd  $\alpha$-continued fractions, on the right for Nakada's $\alpha$-continued fractions. The values in these plots are obtained by simulations.}\label{fig:ent}
\end{figure}

In~\cite{[NN]} it is proven by Nakada and Natsui that for Nakada's $\alpha$-continued fractions there exists decreasing sequences of intervals $(I_n),(J_n),(K_n),$ $(L_n)$ such that $\frac{1}{n}\in I_n$, $I_{n+1}<J_n<K_n<L_N<I_n$ and the entropy of $\hat{T}_\alpha$ is increasing on $I_n$, decreasing on $K_n$ and constant on $J_n$ and $L_n$. Here the map $\hat{T}_\alpha: [\alpha-1,\alpha)\rightarrow[\alpha-1,\alpha)$ is Nakada's $\alpha$-continued fraction map defined as $\frac{1}{|x|}-\left\lfloor \frac{1}{|x|} + 1 - \alpha\right \rfloor$. Do we observe the same phenomena for the Odd $\alpha$-continued fractions? In other words, can we also find decreasing sequences of intervals $(I_n),(J_n),(K_n),$ $(L_n)$ with these properties? Figure~\ref{fig:entzoom} shows the entropy function of both the odd $\alpha$-continued fractions and Nakada's $\alpha$-continued fractions between $\frac{1}{4}$ and $\frac{1}{3}$.

\begin{figure}[h!]
  \centering
{\includegraphics[width=0.45\textwidth]
{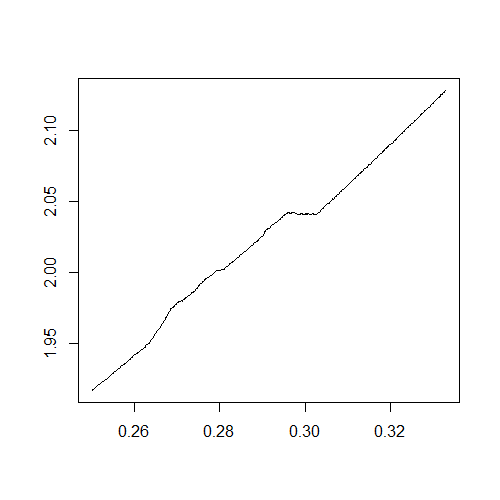}}
  \hfill
{\includegraphics[width=0.45\textwidth]
{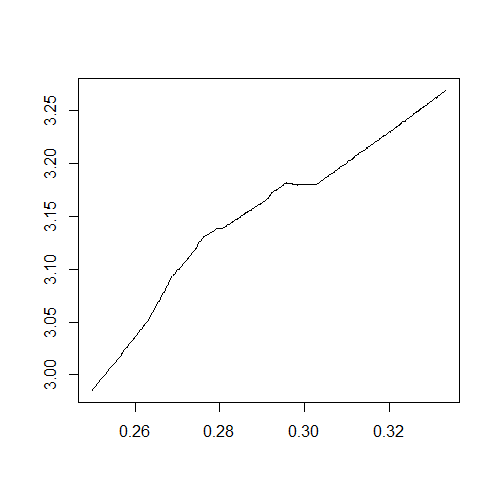}}
  \caption{Simulations of the entropy as a function of $\alpha$ on the interval $[\frac{1}{4},\frac{1}{3}]$. On the left for the odd  $\alpha$-continued fractions, on the right for Nakada's $\alpha$-continued fractions.}\label{fig:entzoom}
\end{figure}

For both types of continued fraction transformations, we observe intervals on which the entropy as a function of $\alpha$ is constant as well as intervals where it is either increasing or decreasing. However the intervals on which the entropy is decreasing are more difficult to spot. A closer look around $\frac{14}{47}$ shows that there are also intervals on which the entropy of $T_\alpha$ is decreasing (see Figure~\ref{fig:entzoom2}).
 
\begin{figure}[h!]
  \centering
{\includegraphics[width=0.45\textwidth]
{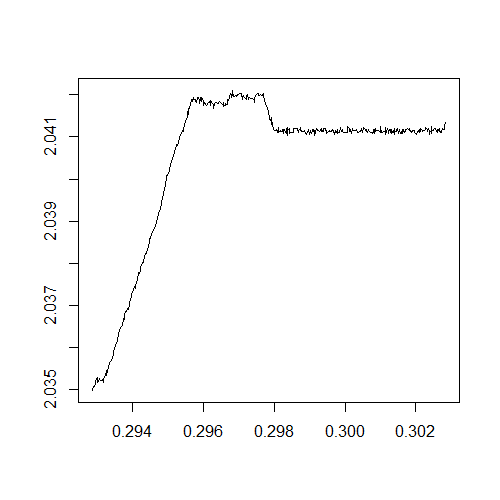}}
  \hfill
{\includegraphics[width=0.45\textwidth]
{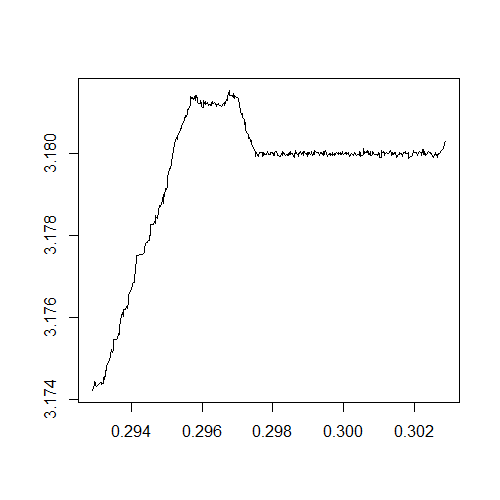}}
  \caption{Simulations of the entropy as a function of $\alpha$ on an interval containing $\frac{14}{47}$. On the left for the odd  $\alpha$-continued fractions, on the right for Nakada's $\alpha$-continued fractions.}\label{fig:entzoom2}
\end{figure}
In analogy with the main result of~\cite{[NN]} for Nakada's $\alpha$-continued fractions we have the following result.

\begin{theorem}\label{theorem}
There exists sequences of intervals $(I_n),(J_n),(K_n),(L_n),n\geq 3$ such that
\begin{enumerate}
\item{$\frac{1}{n}\in I_n$,}
\item{$I_{n+1}<J_n<K_n<L_n<I_n$,}
\item{The entropy of $T_\alpha$ is increasing on $I_n$, decreasing on $K_n$ and constant on $J_n$ and $L_n$ for every $n\geq 3$.}
\end{enumerate}
Here $I<J$ means that any element of $J$ is strictly larger than any element of $I$.
\end{theorem}
We will prove this theorem using a series of lemmas with a prominent role reserved for a phenomena called \textit{matching}. We say that matching holds for a parameter $\alpha$ if there are $N,M\in\mathbb{N}$ such that  
\begin{equation}\label{eqi:matching}
T_{\alpha}^N(\alpha)=T_{\alpha}^M(\alpha-2).
\end{equation}
For the families earlier mentioned it is proven that matching holds for almost every parameter, i.e for almost every $\alpha\in(0,1)$ there are $N,M\in\mathbb{N}$ such that $T_{\alpha}^N(\alpha)=T_{\alpha}^M(\alpha-1)$; see~\cite{[CIT],[CLS],[CT]}. Furthermore, the parameter space breaks up into intervals on which the matching exponents (the minimal numbers $N,M$ such that (\ref{eqi:matching}) holds) are constant. We can choose $(I_n),(J_n),(K_n),(L_n)$ as such intervals (they will not cover the parameter space entirely). On such interval, if $N-M<0$ the entropy of $T_\alpha$ is  increasing, if $N-M=0$ the entropy of $T_\alpha$ is constant and if $N-M>0$ the entropy of $T_\alpha$ is decreasing. It turns out that for our family of continued fractions we can also identify matching intervals and this property also holds. The proof is in analogy to the other families (see~\cite{[L]}).
Before we make sequences of intervals we will first make sequences of rational numbers. For these rational numbers we determine the matching exponents $(N,M)$.  Then we will show that for points in a small neighborhood of these rationals we have matching exponents $(N+1,M+1)$. The sequences are chosen in a way such that the neighborhoods will give rise to the desired intervals $(I_n),(J_n),(K_n),(L_n)$.
\begin{proposition}\label{prop:seq}
Let $n\geq 3$ be an integer and define the following sequences:
\begin{equation}\label{eq:seq}
a_n=\frac{1}{n}, \quad b_n=\frac{2n+1}{2n^2+2n+1},\quad  c_n=\frac{5n-1}{n(5n-1)+5}, \quad d_n=\frac{n}{n^2+1}.
\end{equation}
For these sequences we have $a_{n+1}<b_n<c_n<d_n<a_n$. Furthermore, 
\begin{enumerate}[label=(\roman*)]
\item{for $a_n$ we have that the matching exponents $(N,M)$ are \newline $(1,3\frac{(n-1)}{2}+1)$ for $n$ is odd and $(2,3\frac{(n-2)}{2}+2)$ for $n$ is even.}
\item{for $b_n$ we have that the matching exponents $(N,M)$ are \newline $(3\frac{(n+1)}{2}+1,3\frac{(n+1)}{2}+1)$ for $n$ is odd and $(\frac{3n}{2}+1,\frac{3n}{2}+1)$ for $n$ is even.}
\item{for $c_n$ we have that the matching exponents $(N,M)$ are \newline $(3\frac{(n-1)}{2}+6,3\frac{(n-1)}{2}+3)$ for $n$ is odd and $(\frac{3n}{2}+4,\frac{3n}{2}+1)$ for $n$ is even.}
\item{for $d_n$ we have that the matching exponents $(N,M)$ are \newline $(3\frac{(n-1)}{2}+2,3\frac{(n-1)}{2}+2)$ for $n$ is odd and $(\frac{3n}{2},\frac{3n}{2})$ for $n$ is even.}
\end{enumerate}
\end{proposition}

\begin{proof}
For $\alpha=a_n$, $n$ odd we have $T_{\alpha}(\alpha)=0$. For $\alpha-2$ we find:
\begin{eqnarray*}
T_\alpha(\alpha-2)&=&\frac{n}{2n-1}-1=\frac{1-n}{2n-1}\\
T_\alpha^2(\alpha-2)&=&\frac{2n-1}{n-1}-3=\frac{2-n}{n-1}\\
T_\alpha^3(\alpha-2)&=&\frac{n-1}{n-2}-3=\frac{5-2n}{n-2}=a_{n-2}-2\\
\end{eqnarray*}
(in sase $n=3$ we set $a_{n-2}=1$). One can check that these three iterates of $\frac{1}{n}-2$ are always in the interval $(\frac{1}{n}-2,0)$.
Now
$$
a_n-2 < a_{n-2}-2 < \cdots < a_3 - 2 < -1 < \frac{-n}{n+1}
$$
so that for $k=0,1,\dots, \frac{n-1}{2}$ the first digit of $a_{n-2k} -2$ is $-1$. Since 
\[
T_{\alpha} (a_3 -2) = \tfrac{3}{5} -1 = -\tfrac{2}{5},
\quad T_{\alpha}^2 (a_3 -2) = \tfrac{5}{2} - 3 = -\tfrac{1}{2}, \quad
T_{\alpha}^3 (a_3 -2) = 2 - 3 = -1 
\]
hold, we see that 
\[
T_\alpha^{3\frac{(n-1)}{2}}(\alpha-2)=-1 \quad \text{and so } \quad T_\alpha^{3\frac{(n-1)}{2}+1}(\alpha-2)=0.
\]
In fact, we just saw that the $\alpha$-expansion of 
$\alpha -2$ is 
$$
\alpha-2=[0;(-1,-3,-3)^{\frac{n-1}{2}},-1]
$$
For $\alpha=a_n$, $n$ even we have $T_{\alpha}(\alpha)=-1$ so $T_{\alpha}^2(\alpha)=0$. For $\alpha-2$ the first iterations will give the same digits as for $n$ is odd, so we again find:
\[
T_\alpha^3(\alpha-2)=a_{n-2}-2.
\]
This gives 
\[
T_\alpha^{3\frac{(n-2)}{2}}(\alpha-2)=a_{2}-2=-\frac{3}{2} \quad \text{ and so } \quad T_\alpha^{3\frac{(n-2)}{2}+2}(\alpha-2)=0.
\]
Note that from this calculation we also found the $\alpha$-continued fraction expansion $$\alpha-2=[0;(-1,-3,-3)^{\frac{n-2}{2}},-1,-3].$$
\vspace{1em}

For $\alpha=b_n$, $n$ odd we have 
\begin{eqnarray*}
T_\alpha(\alpha)&=&\frac{2n^2+2n+1}{2n+1}-(n+2)=\frac{-(3n+1)}{2n+1}\\
T_\alpha^2(\alpha)&=&\frac{2n+1}{3n+1}-1=\frac{-n}{3n+1}\\
T_\alpha^3(\alpha)&=&\frac{3n+1}{n}-5=\frac{1}{n}-2=a_{n}-2.\\
\end{eqnarray*}
Using $b_n<a_n$ and that $T_{a_n}^k(a_n-2)\leq 0$ for all $k\in \{1,\ldots, \frac{3(n-1)}{2}+1\}$ we find that, for these values of $k$, that  $T_{a_n}^k(a_n-2)= T_{b_n}^k(a_n-2)$ and so
$$
b_n=[0;-(n+2),-1,-5,(-1,-3,-3)^{\frac{n-1}{2}},-1].$$ This yields that $T_\alpha^{\frac{3(n+1)}{2}+1}(\alpha)=0$.
\vspace{1em}\newline
For $\alpha-2$ let us write $n=2k+1$ and consider $\beta^\prime=[0;(-1,-3-,3)^k,-1]$. We will show that we can write $\alpha-2$ as $\alpha-2= [0;(-1,-3,-3)^k,-1,\ldots]$.  Using the recurrent formula's (\ref{eq:pnrecurrencerelation}) and (\ref{eq:qnrecurrencerelation}) we find for $0\leq j\leq k$,
\[
\begin{array}{l l}
p_{3j-1}(\beta^\prime)=1-4j & q_{3j-1}(\beta^\prime)=2j\\
p_{3j}(\beta^\prime)= -8j & q_{3j}(\beta^\prime)=4j+1\\
p_{3j+1}(\beta^\prime)=-1-4j & q_{3j+1}(\beta^\prime)=2j+1\\
\end{array}
\]
One can prove this by induction. Let $\beta=[0;(-1,-3,-3)^k,-1,-5,-1,-(2k+3)]$. We will show that $\beta=\alpha-2$. We find that the numerators and denominators of the convergents of $\beta$ satisfy:
\[
\begin{array}{l l}
p_{3k}(\beta)=-8k & q_{3k}(\beta)=4k+1\\
p_{3k+2}(\beta)=-12k-5 & q_{3k+2}(\beta)=6k+4\\
p_{3k+3}(\beta)=-8k-4 & q_{3k+3}(\beta)=4k+3\\
p_{3k+4}(\beta)=-16k^2-20k-7 & q_{3k+4}(\beta)=8k^2+12k+5.\\
\end{array}
\]
This gives
\[
\beta=\frac{-16k^2-20k-7}{8k^2+12k+5}
\]
and by substituting $n=2k+1$ in $b_n-2$, as given in (\ref{eq:seq}), we find
\[
b_n-2=\frac{-16k^2-20k-7}{8k^2+12k+5}.
\]
We conclude that $\beta=b_n-2$ and so $[0;(-1,-3,-3)^k,-1,-5,-1,-(2k+3)]$ is an expansion of $b_n-2$. In order to prove this is the odd $\alpha$-continued fraction expansion of $b_n-2$ for $\alpha=b_n$ we need to show that all points in the orbit of this continued fraction are between $\alpha-2$ and $\alpha$. Since all $\varepsilon_i(\alpha-2)=-1$ we know that any point in the orbit is smaller than $0$. Now note that to find the ordering of two continued fractions $a=[0;-a_1,-a_2,\ldots]$ and  $b=[0;-b_1,-b_2,\ldots]$ we look at the first digit that differs and the ordering of these digits is taken (so if $a_i>b_i$ then $a>b$). By repeatedly deleting the first digit of $[0;(-1,-3,-3)^k,-1,-5,-1,-(2k+3)]$  and observing that the remaining continued fraction is are strictly larger than $[0;(-1,-3,-3)^k,-1,-5,-1,-(2k+3)]$ we conclude that the orbit over iterations of $T_\alpha$ must be larger than $\alpha-2$.  Hence, $[0;(-1,-3,-3)^k,-1,-5,-1,-(2k+3)]$ is the $\alpha$-continued fraction of $b_n-2$. In the same manner one can check the other matching indices.
\end{proof}
Note that we found the continued fraction expansions of $\alpha$ and $\alpha - 2$ as well. When doing the computations for the other sequences this is also the case. Table~\ref{tab} gives all the odd $\alpha$-continued fraction expansions.  The values or $\varepsilon_i(x)$ is not given in the table in order to make the table more readable. We have $\varepsilon_1(\alpha)=1$ and $\varepsilon_i(\alpha)=-1$ for every $i\geq 2$ for every $\alpha$. For $\alpha-2$ we have that $\varepsilon_i(\alpha-2)=-1$ for every $i\in\mathbb{N}$ and every choice of $\alpha$ from the table.
\newline
\begin{table}[h!]
  \centering
\begingroup
\renewcommand*{\arraystretch}{2}
\noindent\begin{tabular}{| c | c | c | c | c | r |}
\hline
$\alpha$  & digits of the odd $\alpha$-CF of $\alpha$ & digits of the odd $\alpha$-CF of $\alpha-2$ &  \\
\hline
$a_n$ & $ n $ & $(1,3,3)^{\frac{n-1}{2}},1$  & odd\\
	  & $n+1,1$ & $(1,3,3)^{\frac{n-2}{2}},1,3$  & even\\
\hline 
$b_n$ & $n+2,1,5,(1,3,3)^{\frac{n-1}{2}},1$ & $(1,3,3)^{\frac{n-1}{2}},1,5,1,n+2$  & odd\\
	  & $n+1,(3,3,1)^{\frac{n}{2}}$ & $(1,3,3)^{\frac{n}{2}},n+1$  & even\\
\hline 
$c_n$ & $n+2,(1,3,3)^{\frac{n-1}{2}},3,1,3,3,1$ & $(1,3,3)^{\frac{n-1}{2}},1,n+2,5$  & odd\\
	  & $n+1,3,(1,3,3)^{\frac{n-2}{2}},3,1,3,3,1$ & $(1,3,3)^{\frac{n-2}{2}},1,3,n+1,5$  & even\\
\hline 
$d_n$ & $n+2,(1,3,3)^{\frac{n-1}{2}},1$ & $(1,3,3)^{\frac{n-1}{2}},1,n+2$  & odd\\
	  & $n+1,3,(1,3,3)^{\frac{n-2}{2}},1$ & $(1,3,3)^{\frac{n-2}{2}},1,3,n+1$  & even\\
\hline
\end{tabular}
\endgroup \caption{The digits of the odd $\alpha$-continued fractions expansions of  $\alpha$ and $\alpha-2$ for $\alpha$ from the sequences of (\ref{eq:seq}).}\label{tab}
\end{table}

\subsection{Algebraic relations and corollaries thereof}
We will now show that certain algebraic relations hold for the rationals in our sequences. 
To find algebraic relations we work with M\"obius transformations and matrices.
Let 
\[
A=\left[
\begin{array}{c c}
 e_1 & e_2\\
 e_3 & e_4\\
\end{array}
\right]
\]
be a matrix with $e_i\in\mathbb{Z}$. 
The M\"obius transformation induced by $A$ is the map $A:\mathbb{R}\rightarrow\mathbb{R}$ given by
\begin{equation}\label{eq:XX}
A(z)=\frac{e_1z+e_2}{e_3z+e_4}.
\end{equation}
\noindent Now let $d\in 2\mathbb{N}-1$ and $\varepsilon=\pm 1$. We define the following matrices in $\text{SL}_2(\mathbb{Z})$
\[
B_{\varepsilon,d}=\left[
\begin{array}{c c}
 0 & \varepsilon\\
 1 & d\\
\end{array}
\right]
, \quad
R=\left[
\begin{array}{ c c }
 1 & 1\\
 0 & 1\\
\end{array}
\right]
, \quad
S=\left[
\begin{array}{ c c }
 0 & 1\\
 1 & 0\\
\end{array}
\right]
, \quad
V=\left[
\begin{array}{ c c }
  -1 & 0\\
   0 & 1\\
\end{array}
\right] .
\] 
Fix $\alpha$ and $x\in [\alpha-2,\alpha)$ and let
\[
M_{\alpha,x,n}=B_{\varepsilon_{\alpha,1}(x),d_{\alpha,1}(x)}B_{\varepsilon_{\alpha,2}(x),d_{\alpha,2}(x)}B_{\varepsilon_{\alpha,3}(x),d_{\alpha,3}(x)}\cdots B_{\varepsilon_{\alpha,n}(x),d_{\alpha,n}(x)}.
\]
One can check shows that 
\begin{equation}\label{eq:XXX}
M_{\alpha,x,n}=\left[
\begin{array}{ c c }
 p_{\alpha,n-1}(x) & p_{\alpha,n}(x)\\
 q_{\alpha,n-1}(x) & q_{\alpha,n}(x)\\
\end{array}
\right]
\end{equation}
\begin{lemma}\label{lem:forseq}
For all $\alpha$ from the sequences of (\ref{eq:seq})  with $(N,M)$ minimal such that $T_\alpha^N(\alpha)=T_\alpha^M(\alpha-2)=0$ we have 
\begin{equation}\label{eq:alg2}
M_{\alpha,\alpha,N}=R^2M_{\alpha,\alpha-2,M}VSR^2S.
\end{equation}
\end{lemma}
Equation (\ref{eq:alg2}) will turn out to be very helpful later. For Nakada's $\alpha$-continued fractions a similar equation holds for rational points that are of importance in that setting. In that case the equation is given by $M_{\alpha,\alpha,N}=RM_{\alpha,\alpha-2,M}VSRS$; see \cite{[CT]}. Note that $R(x)=x+1$ and that in our case the endpoints of the interval on which we defined $T_\alpha$ are $\alpha$ and $\alpha-2$ instead of $\alpha$ and $\alpha-1$.
\begin{proof}
Now let $\alpha$ be from the sequences of (\ref{eq:seq}) with $(N,M)$ minimal such that $T_\alpha^N(\alpha)=T_\alpha^M(\alpha-2)=0$. From Table~\ref{tab} one can observe that these exponents are the matching exponents. Furthermore, we have $\frac{p}{q}=\frac{p_{\alpha,N}(\alpha)}{q_{\alpha,N}(\alpha)}$ and $\frac{p}{q}-2=\frac{p_{\alpha,M}(\alpha-2)}{q_{\alpha,M}(\alpha-2)} $.
This gives
\[
M_{\alpha,\alpha,N}= 
\left[
 \begin{array}{ c c }
  p_{\alpha,N-1}(\alpha) & p\\
  q_{\alpha,N-1}(\alpha) & q\\
 \end{array}
 \right]
\quad
M_{\alpha,\alpha-2,M}=
\left[
 \begin{array}{ c c }
   p_{\alpha,M-1}(\alpha-2) & p-2q\\
   q_{\alpha,M-1}(\alpha-2) & q\\
 \end{array}
 \right].
\] 
To be able to verify (\ref{eq:alg2}) we need the second to last convergent of $\alpha$ and $\alpha-2$. By using the continued fractions in Table~\ref{tab} and the recurrence relations (\ref{eq:pnrecurrencerelation}) and (\ref{eq:qnrecurrencerelation}) we can compute these second to last convergent of $\alpha$ and $\alpha-2$. This results in Table~\ref{tabalg}.

\begin{table}[h!]
  \centering
\begingroup
\renewcommand*{\arraystretch}{1.5}

\noindent\begin{tabular}{| c  c  c  r |}
\hline
$\alpha$ & $M_{\alpha,\alpha,N}$ & $M_{\alpha,\alpha-2,M}$ &   \\
\hline
$a_n$ & 
$\left[
 \begin{array}{ c c }
  0 & 1\\
 1 & n\\
 \end{array}
 \right]$ &
$\left[
 \begin{array}{ c c }
   4-4n & 1-2n\\
   2n-1 & n\\
 \end{array}
 \right] $ & odd\\[1.5em]
	  & $\left[
 \begin{array}{ c c }
  1 & 1\\
  n+1 & n\\
 \end{array}
 \right]$ &
 $\left[
 \begin{array}{ c c }
   3-2n & 1-2n\\
   n-1 & n\\
 \end{array}
 \right] $  & even\\[1.5em]
$b_n$ & $\left[
 \begin{array}{ c c }
  4n & 2n+1\\
  4n^2+2n+1 & 2n^2+2n+1\\
 \end{array}
 \right]$  & 
 $\left[
 \begin{array}{ c c }
   -4n & -4n^2-2n-1\\
  2n+1 & 2n^2+2n+1\\
 \end{array}
 \right] $ & odd and even\\[1.5em]
$c_n$ & $\left[
 \begin{array}{ c c }
  9n-2 & 5n-1\\
9n^2-2n+9 & n(5n-1)+5\\
 \end{array}
 \right]$  & 
 $\left[
 \begin{array}{ c c }
-2n^2+n-2 & 7n-10n^2-11\\
n^2+1 & n(5n-1)+5\\
 \end{array}
 \right] $ & odd and even\\[1.5em]
$d_n$ & $\left[
 \begin{array}{ c c }
  2n-1 & n\\
  2n^2-n+2 & n^2+1\\
 \end{array}
 \right]$  & 
 $\left[
 \begin{array}{ c c }
   1-2n & n-2n^2-2\\
  n & n^2+1\\
 \end{array}
 \right] $ & odd and even\\[1.5em]
\hline
\end{tabular}
\endgroup \caption{$M_{\alpha,\alpha,N}$ and $M_{\alpha,\alpha-2,M}$ for the sequences from (\ref{eq:seq}). These are found by calculating the convergents using Table~\ref{tab}.}\label{tabalg}
\end{table}
By substituting these matrices in (\ref{eq:alg2}) we find that the equation indeed  holds for any $\alpha$ from the sequences of (\ref{eq:seq}).
\end{proof}
\begin{lemma}\label{lem:neighbor}
If (\ref{eq:alg2}) holds for $\alpha$ and $(N,M)$ are the matching exponents then we have matching in a sufficiently small neighborhood of $\alpha$ with matching exponents $(N+1,M+1)$. Furthermore, let $x$ be a point in such neighborhood. Then
\begin{equation}\label{eq:alg1}
M_{x,x,N+1}=R^2M_{x,x-2,M+1}
\end{equation} 
holds for $x$.
\end{lemma}
\begin{proof}
Fix $\alpha\in\mathbb{Q}\cap(0,1)$ such that (\ref{eq:alg2}) holds for $\alpha$ with matching exponents $(N,M)$ and $x$ in a sufficiently small neighborhood of $\alpha$. Using (\ref{eq:XX}) and (\ref{eq:XXX}) we can write $x$ as 
\begin{equation}\label{eq:writx1}
x=M_{x,x,N}(T_x^N(x))
\end{equation}
and
\begin{equation}\label{eq:writx2}
x=R^2M_{x,x-2,M}(T_x^M(x-2)) .
\end{equation}
Since $x$ is in a sufficiently small neighborhood of $\alpha$ we have that 
\[
M_{x,x,N}=M_{\alpha,\alpha,N} \text{\  and \ } M_{x,x-2,M}=M_{\alpha,\alpha-2,M} .
\]
Equations (\ref{eq:alg2}), (\ref{eq:writx1}) and (\ref{eq:writx2}) give 
\[
VSR^2ST_x^N(x)=
\left[
 \begin{array}{ c c }
  -1 & 0\\
  2 & 1\\
\end{array}
 \right] \left(T_x^N(x)\right) =T_x^M(x-2)
\]
and so
\begin{equation}\label{eq:TN}
T_x^M(x-2)=\frac{-T_x^N(x)}{2T_x^N(x)+1}
\end{equation}
which implies
\[
T_x^N(x)=\frac{-T_x^M(x-2)}{2T_x^M(x-2)+1} .
\]
Now
\begin{eqnarray*}
T_x^{M+1}(x-2)&=&\left|\frac{2T_x^N(x)+1}{-T_x^N(x)}\right|-d_{x,M+1}(x-2),\\
T_x^{N+1}(x)&=&\left|\frac{2T_x^M(x-2)+1}{-T_x^M(x-2)}\right|-d_{x,N+1}(x).\\
\end{eqnarray*}
This gives
\[
T_x^{M+1}(x-2)-T_x^{N+1}(x)=\left|\frac{-1}{T_x^N(x)}-2\right|-\left|\frac{-1}{T_x^M(x-2)}-2\right|+d_{x,N+1}(x)-d_{x,M+1}(x-2).
\]
Using (\ref{eq:TN}) we find
\[
T_x^{M+1}(x-2)-T_x^{N+1}(x)=\left|\frac{-1}{T_x^N(x)}-2\right|-\left|\frac{1}{T_x^N(x)}\right|+d_{x,N+1}(x)-d_{x,M+1}(x-2).
\]
When $T_x^N(x)>0$ we find
\[
T_x^{M+1}(x-2)-T_x^{N+1}(x)=\frac{1}{T_x^N(x)}+2-\frac{1}{T_x^N(x)}+d_{x,N+1}(x)-d_{x,M+1}(x-2).
\]
When $T_x^N(x)<0$, since $x$ is in a neighborhood of $\alpha$,  $T_x^N(x)$ is in the neighborhood of $0$ and so
\[
T_x^{M+1}(x-2)-T_x^{N+1}(x)=-\frac{1}{T_x^N(x)}-2+\frac{1}{T_x^N(x)}+d_{x,N+1}(x)-d_{x,M+1}(x-2).
\]
In both cases due to the fact that the digits are odd we have 
\[
T_x^{N+1}(x)-T_x^{M+1}(x-1)=r 
\]
for some $r\in2\mathbb{Z}$. Since $T_x^{N+1}(x), T_x^{M+1}(x-2) \in [x-2,x)$ we find $r\in 2\mathbb{Z} \cap (-2,2) =\{0\}$ and therefore we have $T_x^{N+1}(x)=T_x^{M+1}(x-1)$.
Furthermore, since $x=M_{x,x,N+1}(T_x^{N+1}(x))$ and $x=R^2M_{x,x-2,M+1}(T_x^{M+1}(x-2))$ we have that (\ref{eq:alg1}) holds.
\end{proof} 

\begin{proof}[Proof of Theorem \ref{theorem}] Note that by using Lemma~\ref{lem:forseq} and \ref{lem:neighbor}  we found a matching interval for every $\alpha$ from the sequences of (\ref{eq:seq}) for which equation (\ref{eq:alg1}) hold for all elements in this interval. This equation is necessary and sufficient to prove that on such interval, if $N-M<0$ the entropy is increasing, if $N-M=0$ the entropy is constant and if $N-M>0$ the entropy is decreasing. The proof is the same as for the other families mentioned with the only adaptation that $\alpha-1$ is replaced by $\alpha-2$; see \cite{[L]}, Section 3. Furthermore, from Proposition \ref{prop:seq} we find that for $\alpha=a_n$ we have $N-M<0$, for $b_n$ and $d_n$ we have $N-M=0$ and for $c_n$ we have $N-M>0$. This finishes the proof.
\end{proof}
An interesting observation is that, in order to prove the same statement for Nakada's $\alpha$-continued fractions, one can use the same sequences. Though one can check that we do not have the same intervals of monotonicity of the entropy for both families. Note that, in this paper, we proved matching holds for a countable set of intervals. These intervals do not cover the entire parameter space. Two large matching intervals that do cover a large proportion are $(g,1)$ and $(1,G)$. 
For the families Nakada's $\alpha$-continued fractions, Ito and Tanaka's $\alpha$-continued fractions and Katok and Ugarcovici's $\alpha$-continued fractions~\cite{[KUa],[KUb]} it is proven that matching holds almost everywhere (see \cite{[CIT],[CLS],[CT]}). Since the new family behaves in a similar way concerning matching we conjecture the following:
\begin{conjecture}
Matching also holds for Lebesgue almost every $\alpha\in[0,g]$. 
\end{conjecture}
Even though it would surprise us if the following conjecture would be false, we were unable to prove it at this point.
In order to prove this a more precise description of the set of $\alpha$ for which matching does not hold might be needed.


\begin{thebibliography}{[NIT]}

\bibitem{[BM]} Boca, Florin P., and Merriman, Claire -- $\alpha$-expansions with odd partial quotients. J.\ Number Theory 199, 322-341 (2019).

\bibitem{[BKS]}  Burton, Robert M.; Kraaikamp, Cornelis, and Schmidt, Thomas A. --  Natural extensions for the Rosen fractions. Trans. Amer. Math. Soc. 352 (2000), no. 3, 1277–1298. 

\bibitem{[CIT]}  Carminati, Carlo;  Isola, Stefano, and  Tiozzo, Giulio -- Continued fractions with $SL(2,\mathbb{Z})$-branches:  combinatorics and entropy. Trans.Amer. Math. Soc., 370(7):4927–4973, 2018.


\bibitem{[CLS]} Carminati, Carlo; Langeveld, Niels, and Steiner, Wolfgang -- Tanaka-Ito $\alpha$-continued fractions and matching. 2020. \textit{to appear in Nonlinearity } \url{https://arxiv.org/pdf/2004.14926.pdf}.

\bibitem{[CT]}  Carminati, Carlo, and Tiozzo, Giulio -- A canonical thickening of $\mathbb{Q}$ and the entropy of $\alpha$-continued fraction transformations. Ergodic Theory Dynam. Systems 32 (2012), no. 4, 1249–1269.


\bibitem{[DK]} Dajani, Karma, and Kraaikamp, Cor -- Ergodic theory of numbers. Carus Mathematical Monographs, 29. Mathematical Association of America, Washington, DC, 2002.

\bibitem{[HK]} Hartono, Yusuf, and Kraaikamp, Cor -- On continued fractions with odd partial quotients. Rev.\ Roum.\ Math.\ Pures Appl.\ 47, No. 1, 43-62 (2002).

\bibitem{[IK]} Iosifescu, Marius, and Kraaikamp, Cor -- Metrical Theory of Continued Fractions, Mathematics and its Applications, vol. 547, Kluwer, Dordrecht, 2002.

\bibitem{[KSS]} Kraaikamp, Cor; Schmidt, Thomas A., and Steiner, Wolfgang -- Natural extensions and entropy of $\alpha$-continued fractions,
Nonlinearity 25 (2012), no.\ 8, 2207–2243.

\bibitem{[KUa]}   Katok, Svetlana,  and    Ugarcovici, Ilie  --   Structure  of  attractors  for $(a,b)$-continued fraction transformations. J. Mod. Dyn., 4(4):637–691, 2010.

\bibitem{[KUb]}   Katok, Svetlana,  and    Ugarcovici, Ilie  --  Theory of $(a,b)$-continued fraction  transformations  and  applications. Electron.  Res.  Announc.Math. Sci., 17:20–33, 2010.

\bibitem{[L]}  Langeveld, Niels  -- Matching, entropy, holes and expansions. PhD thesis, Leiden Universtity, 2019. \url{https://hdl.handle.net/1887/81488}

\bibitem{[LM]}  Luzzi, Laura, and Marmi, Stefano  -- On the entropy of Japanese continued fractions. Discrete Contin. Dyn. Syst. 20 (2008), no. 3, 673–711.


\bibitem{[N]} Nakada, Hitoshi - Metrical theory for a class of continued fraction transformations and their natural extentions, Tokyo J.\ of Math. 4 (1981), 399--429.

\bibitem{[NN]} Nakada, Hitoshi, and Natsui, Rie. -- The non-monotonicity of the entropy of {$\alpha$}-continued fraction transformations, Nonlinearity 6 (2008), 1207--1225.

\bibitem{[Rie]} Rieger, Georg.\ J.\ - On the metrical theory of continued fractions with odd partial quotients. Topics in classical number theory, Vol. I, II (Budapest, 1981), 1371--1418, Colloq. Math. Soc. J\'anos Bolyai, 34, North-Holland, Amsterdam-New York, 1984.

\bibitem{[R]} Rosen, David - A class of continued fractions associated with certain properly discontinuous groups.
Duke Math. J. 21 (1954), 549–563. 30.0X.

\bibitem{[S1]} Schweiger, Fritz. - Continued fractions with odd and even partial quotients, Arbeitbericht Mathematisches Institut Salzburg 4 (1982), 59--70.

\bibitem{[S2]} Schweiger, Fritz. - On the approximation by continued fractions with odd and even partial quotients, Arbeitbericht Mathematisches Institut Salzburg 1-2 (1984), 105--114.

\bibitem{[Se1]} Sebe, Gabriela I. - Gauss' problem for the continued fraction with odd partial quotients, Monatshefte f.\ Math.\ 133, No. 3, 241-254 (2001).

\bibitem{[Se2]} Sebe, Gabriela I. - On Convergence Rate in the Gauss-Kuzmin Problem for Grotesque Continued Fractions, Revue Roumaine
Math. Pures Appl.\ 46, No.\ 6, 839-852 (2001).

\bibitem{[TI]} Tanaka, Shigeru, and Ito, Shunji -- On a family of continued fraction transformations and their ergodic properties. Tokyo J.\ Math.~4, 153--175 (1981).


\end{thebibliography}
\end{document}